\documentclass[reqno, 11pt]{amsart}
\usepackage{amsmath}
\usepackage{framed,comment,enumerate}
\usepackage[pdftex]{graphicx}
\usepackage{epstopdf}
\usepackage{xcolor, framed}
\usepackage{color}

\usepackage{amsmath, amsthm, amssymb, esint}
\usepackage{amsfonts}
\usepackage{bbm}
\usepackage[ansinew]{inputenc}
\usepackage[dvips]{epsfig}
\usepackage{graphicx}
\usepackage[english]{babel}
\usepackage{enumerate}
\usepackage{enumitem}
\usepackage{hyperref}
\usepackage{fbox}

\def\R {\mathbb{R}}
\def\N{\mathbb{N}}
\def\eps{\varepsilon}
\def\Sph{\mathbb{S}^{d-1}}

\def\FlatC{(x_d)_+}

\def\DimH{\mathrm{dim}_{\mathcal{H}}}

\def\En{\mathcal{E}}
\def\Diri{|\nabla u|^2}

\def\PosS{\{u>0\}}
\def\ConS{\{u=0\}}

\def\POSS{\{U>0\}}

\def\GU{\Gamma(u)}
\def\Reg{\mathrm{Reg}}
\def\Sing{\mathrm{Sing}}
\def\SingU{\mathrm{Sing}(u)}

\def\spt{\mathrm{spt}}

\def\FamPhi{\{g_t\}_{t\in(-1,1)}}

\def\C{\mathcal{C}(\R^d)}
\def\SC{\mathcal{SC}(\R^d)}
\def\RS{\mathcal{R}}

\def\M{\mathcal{M}}
\def\OM{\mathcal{OM}}

\def\hem{\hspace{0.5em}}
\def\vem{\vspace{0.6em}}

\newtheorem{thm}{Theorem}[section]
\newtheorem{prop}[thm]{Proposition} 

\newtheorem{ass}{Assumption} 
\newtheorem{conj}{Conjecture}
\newtheorem{op}{Open Question}

\newtheorem{cor}[thm]{Corollary}
\newtheorem{lem}[thm]{Lemma}
\theoremstyle{definition}
\newtheorem{defi}[thm]{Definition}
\numberwithin{equation}{section}

\theoremstyle{remark}
\newtheorem{rem}[thm]{Remark}

\title[Linearized equation and generic regularity]{Linearized equation and  generic regularity in the Alt-Caffarelli problem} 

\author{Xavier Fern\'andez-Real}
\address{Institute of Mathematics, \'Ecole Polytechnique F\'ed\'erale de Lausanne, Lausanne, Switzerland}
\email{xavier.fernandez-real@epfl.ch}

\author{Hui Yu}
\address{Department of Mathematics,	National University of Singapore, Singapore}
\email{huiyu@nus.edu.sg}

\thanks{X. F. was supported by the Swiss National Science Foundation (SNF grant PZ00P2\_208930),   by the Swiss State Secretariat for Education, Research and Innovation (SERI) under contract number MB22.00034, and by the AEI project PID2021-125021NA-I00 (Spain).}

\begin{document}

\begin{abstract}
For the Alt-Caffarelli problem, we study  free boundary regularity  of energy minimizers. In six dimensions, we show that free boundaries are analytic for generic boundary data. In general, we improve previous generic Hausdorff dimensions of the singular sets. 

To achieve this, we analyze positive solutions to the linearized equation around homogeneous minimizers (possibly with singular sections on the sphere). For this equation,  we prove a Harnack inequality and establish a dimensional lower bound for its principal eigenvalue. 
\end{abstract}

\maketitle
\tableofcontents
%%%%%%%%%%%%%%%%%%%%%%%%%%%%%%%%%%%%%%%%%%%%%%%%%%%%%%%%%%%%%%%%%%%%%%%%%%%%%%%%%%%%%%%%%%%%%%%%%%%%%%%%%%%%%%%%%%%%%%%%%%%%%%%%%%%%%%%%%%%%%%%%%%%%%%%%%%%%%%%%%%%%%%%%%%%%%%%%%%%%%%%%%%%%%
\section{Introduction}
For a nonnegative function $u$ on a domain $\Omega\subset\R^d$, its \textit{Alt-Caffarelli energy} is given by 
\begin{equation}
\label{EqnAC}
\En(u;\Omega):=\int_\Omega\Diri+\chi_{\PosS},
\end{equation} 
where the characteristic function of a set $E$ is denoted by $\chi_E$.  
This functional was introduced by Alt and Caffarelli in \cite{AC} to address questions in fluid mechanics \cite{BSS, GLS, KW1, KW2}. Since then, the Alt-Caffarelli problem has become one of the most intensely studied free boundary problems. For more background, the reader may consult the monographs by Caffarelli-Salsa \cite{CS} and by Velichkov \cite{V}.  

Under reasonable assumptions on boundary data, it is not difficult to show the existence\footnote{The minimizer may be nonunique for a given boundary datum. However, it is unique for generic data \cite{FeY, FeG}.} of a minimizer $u$ of  \eqref{EqnAC} as well as its optimal regularity \cite{AC}. Much subtler is the regularity of its \textit{free boundary}, namely, 
\begin{equation}
\label{EqnGU}
\GU:=\partial\PosS.
\end{equation} 
This is the interface between the \textit{positive set} $\PosS$ and the \textit{contact set} $\ConS$.

\vem

Following either the classical approach  \cite{AC} or the modern approach \cite{D1}, the free boundary of a minimizer $u$ is decomposed into a \textit{regular part} and a \textit{singular part}
\begin{equation*}
\label{EqnRegAndSing}
\GU=\Reg(u)\cup\Sing(u).
\end{equation*} 
While the regular part is  analytic \cite{C1, C2, CS, DS, KN}, not much is known about $\SingU$ despite exciting developments \cite{ESV}.

One central question is the size of the singular part, for instance, in terms of its Hausdorff dimension $\DimH(\SingU)$. With the monotonicity formula in  Weiss \cite{W}, this reduces to determining the \textit{critical dimension}
\begin{equation}
\label{EqnCriticalDimension}
d^*:=\max\{d\in\N:\hem \text{homogeneous minimizers are rotations of }(x_d)_+ \text{ in }\R^d\}.
\end{equation} 
In this work, homogeneous minimizers are  referred to as \textit{minimizing cones}, and those that are not rotations of $(x_d)_+$ are called \textit{singular minimizing cones}.

If  $u$ is a minimizer  in a domain $\Omega\subset\R^d$,  it follows from \cite{W} that,  for  $d\le d^*$, the singular part is empty and the free boundary is analytic. For $d> d^*$ instead, we have
\begin{equation}
\label{DimensionBound}
\DimH(\SingU)\le d-d^*-1.
\end{equation}
The value of $d^*$ also carries important information for Bernstein-type results \cite{CFFS, D2, EFeY, KaWa}, and currently, we only have partial information on it. That is,
\begin{equation}
\label{EqnPartialInformation}
4\le d^*\le 6.
\end{equation}
The lower bound is due to  Caffarelli-Jerison-Kenig \cite{CJK} and Jerison-Savin \cite{JS}, whereas the upper bound is due to the De Silva-Jerison cone\footnote{Constructing singular minimizing cones is a challenging task. Currently, the De Silva-Jerison cone remains the only known example. 

For stable cones, the reader may refer to the work of Hong \cite{H}. A recent work of Hines-Kolesar-McGrath \cite{HKM} gives examples of homogeneous critical points for \eqref{EqnAC}. 

Interesting non-homogeneous critical points have been constructed in \cite{BSS, HHP, KW1, KW2, LWW}.}  in $\R^7$ \cite{DJ}. 

The following is then a major conjecture for this problem (see \cite{DJ}):
\begin{conj}
\label{Conj}
The critical dimension $d^*$ is $6$.  In particular, for any minimizer $u$ of   \eqref{EqnAC} in  $\Omega\subset\R^d$ with $d\le 6$, we have that  its free boundary $\GU$ is analytic. 
\end{conj}
Despite its fundamental importance, the progress towards this conjecture has been slow, apart from the aforementioned works \cite{CJK, DJ, JS}. 

In this article, we show that \textbf{this conjecture holds for generic data} (see Corollary~\ref{CorMainIntro} or Remark~\ref{RemSharpGamma}).

%%%%%%%%%%%%%%%%%%%%%%%%%%%%%%%%%%%%%%%%%%%%%%%%%%%%%%%%%%%%%%%%%%%%%%%%%%%%%%%%%%%%%%%%%%%%%%%%%%%%%%%%%%%%%%%
\subsection{Generic regularity for the Alt-Caffarelli problem}
In several areas, it has been observed that singularities may be removed by small perturbations of data. In minimal surface theory, we have the classical results of Hardt-Simon \cite{HS} and Smale \cite{Sm}, as well as exciting recent developments by Chodosh-Liokumovich-Spolaor \cite{CLS}, Chodosh-Mantoulidis-Schulze \cite{CMS1,CMS2}, Chodosh-Mantoulidis-Schulze-Wang \cite{CMSWa} and Li-Wang \cite{LWa}. For the obstacle problem, this was conjectured by Schaeffer \cite{Sc}. Following the earlier work by Monneau \cite{M} in $\R^2$, this conjecture was resolved by Figalli, Ros-Oton and Serra \cite{FRS} up to $d=4$. For the Signorini problem, the generic regularity of the free boundary was established by Fern\'andez-Real and Ros-Oton \cite{FeR} and Fern\'andez-Real and Torres-Latorre \cite{FeT} (see also \cite{CC24}). Our work is inspired by these pioneering results. 

For our problem \eqref{EqnAC}, De Silva-Jerison-Shahgholian \cite{DJS} and Edelen-Spolaor-Velichkov \cite{EdSV} constructed families of minimizers with analytic free boundaries near a singular minimizing cone. This hints at a similar phenomenon as described in the previous paragraph. 

The authors  of the present manuscript established a generic regularity result for minimizers of the Alt-Caffarelli energy $\En(\cdot)$ in  \eqref{EqnAC} \cite{FeY}, that improved the estimates for generic data by one dimension. Recall the critical dimension $d^*$ from \eqref{EqnCriticalDimension}. 
\begin{thm}[Theorem 1.5 in \cite{FeY}, Theorem 1.4 in \cite{FeG}]
\label{ThmOldThm}
Let $\{g_t\}_{t\in(-1,1)}$ be an admissible family of boundary data\footnote{The family is required to be continuous for each time and   increasing on $\partial B_1$ with respect to $t$ at a linear rate \cite{FeG}.} on $\partial B_1\subset\R^d$. For each $t$, let $u_t$ be a minimizer of \eqref{EqnAC}  with $u_t|_{\partial B_1}=g_t$. Then, if $d=d^*+1$, there is a countable set $J\subset(-1,1)$ such that $$
\Sing(u_t)=\emptyset\hem \text{ for }t\in (-1,1)\backslash J.
$$
\end{thm} 
In higher dimensions, the estimate in \eqref{DimensionBound} was also improved by $1$ for generic data. 

With the lower bound on $d^*$ in \eqref{EqnPartialInformation}, this shows that, under small perturbations,  free boundaries are analytic in $\R^5$. This misses one dimension in Conjecture~\ref{Conj} for generic data. 

\vem

To close this gap, we note that there are two main ingredients behind Theorem~\ref{ThmOldThm}. With the notation in this theorem, the first ingredient is an estimate on the Hausdorff dimension of $\bigcup_{t\in(-1,1)}\Sing(u_t)$. In this direction, we obtained the sharp estimate in \cite{FeY}. (See Proposition~\ref{PropOptimalResultFeY}.)

The second ingredient is a `cleaning'  lemma, which states that free boundaries separate when we increase the boundary data. In our previous work, we achieved a linear rate of separation (Lemma 4.3 in \cite{FeY}), namely, 
\begin{equation}
\label{EqnOldCleaning}
\mathrm{dist}(\Gamma(u_t),\Gamma(u_s))\ge c|t-s|.
\end{equation}
This allowed us to reduce the dimension of the singular set by one for generic data. 

While this estimate \eqref{EqnOldCleaning} gives the sharp separation between entire free boundaries $\Gamma(u_t)$ and $\Gamma(u_s)$, faster separation is expected between the singular sets $\Sing(u_t)$ and $\Sing(u_s)$. 

\vem

In this work, we obtain this faster separation between singular sets. 
To be precise, we impose the following assumption on boundary data:
\begin{ass}
\label{Ass}
A family of non-negative functions $\{g_t\}_{t\in(-1,1)}\subset C(\overline{B_1})\cap H^1(B_1)$ is admissible if 
$$
g_t-g_s\ge t-s \hem\text{ on }\partial B_1\cap\{g_s>0\}\qquad \text{for all} \quad -1 < s < t < 1. 
$$
\end{ass}
\begin{rem}
Compared with \cite{FeY, FeG}, we have weakened the monotonicity assumption.  Previously, we required the separation on the entire $\partial B_1$. Here it is only imposed on the positive set of the smaller function. This allows boundary data of the form $(h_t)_+$, where $h_t$ is  linearly increasing with respect to  $t$ (not necessarily nonnegative). In particular, we allow nontrivial free boundaries on $\partial B_1.$ 

This opens the door to the study of generic behavior of the free boundary near the  fixed boundary. See, for instance, \cite{ChS,FSV}.
\end{rem} 

For each $t\in(-1,1)$, suppose that $u_t$ denotes a minimizer of \eqref{EqnAC} that takes $g_t$ as boundary data, we can improve the cleaning estimate \eqref{EqnOldCleaning} to 
\begin{equation}
\label{EqnNewCleaningIntro}
\mathrm{dist}(\Sing(u_t),\Sing(u_s))^{1+\gamma_d}\ge c|t-s|
\end{equation} 
for a dimensional $\gamma_d>0$. (See Lemma~\ref{LemCleaningNew}.)  This leads to  the main result of this work:
\begin{thm}
\label{ThmMainIntro}
Let  $\FamPhi$ be an admissible family as in Assumption~\ref{Ass}. For each $t\in(-1,1)$, suppose that $u_t$ is a minimizer of  \eqref{EqnAC} in $B_1$ with 
$$
u_t=g_t \hem\text{ on }\partial B_1.
$$
Then we have, for $d^*$ given by \eqref{EqnCriticalDimension},
\begin{enumerate}
\item{
If $d=d^*+1 \text{ or }d^*+2$, then
$$
\Sing(u_t)=\emptyset \hem\text{ for almost every }t\in(-1,1);
$$
}
\item{
If $d\ge d^*+3$, then
$$
\DimH(\Sing(u_t))\le d-d^*-2-\gamma_d \hem\text{ for almost every }t\in(-1,1),
$$
where $\gamma_d>0$ is a dimensional constant. }
\end{enumerate}
\end{thm} 

As a consequence, we have
\begin{cor}
\label{CorMainIntro}
For a nonnegative $g\in C(\overline{B_1})\cap H^{1}(B_1)$ and a constant $\eps>0$, there is $\tilde{g}$ satisfying
$$
\|\tilde{g}-g\|_{L^{\infty}( \overline{B_1})}+\|\tilde{g}-g\|_{H^1(B_1)}<\eps
$$
such that if $\tilde{u}$ is a minimizer of \eqref{EqnAC} in $B_1$ with $\tilde{u}=\tilde{g}$ on $\partial B_1$, then we have
\begin{enumerate}
\item{
If $d=d^*+1 \text{ or }d^*+2$, then
$$
\Sing(\tilde{u})=\emptyset;
$$
}
\item{
If $d\ge d^*+3$, then
$$
\DimH(\Sing(\tilde{u}))\le d-d^*-2-\gamma_d,
$$
where $\gamma_d>0$ is a dimensional constant. }
\end{enumerate}
\end{cor}
\begin{rem}
With the lower bound in \eqref{EqnPartialInformation}, this resolves Conjecture~\ref{Conj} for generic boundary data. 
\end{rem} 

\begin{rem}
\label{RemSharpGamma}
The constant $\gamma_d$ is given in \eqref{EqnGammaD}. It arises from our dimensional lower bound on the  principal eigenvalue of the linearized equation in Theorem~\ref{ThmDimensionalLowerBound}.  The sharp value of $\gamma_d$ is an important open question that requires new insights. 
\end{rem} 

\begin{rem}
In our previous work \cite{FeY}, Theorem~\ref{ThmOldThm} was established for the Alt-Phillips problem, a family of free boundary problems that includes \eqref{EqnAC} as a special case \cite{AP}. For this family of problems, we expect it is possible to make improvements as in Theorem~\ref{ThmMainIntro} and Corollary~\ref{CorMainIntro}. The new `cleaning' estimate \eqref{EqnNewCleaningIntro}, however,  becomes more challenging for the Alt-Phillips problem, as the linearized equation is more involved \cite{CT, KS, SY}.
\end{rem} 

\vem

The improvement in Theorem~\ref{ThmMainIntro} and Corollary~\ref{CorMainIntro} follows from the superlinearity in \eqref{EqnNewCleaningIntro}. To gain this extra power, we need to analyze the linearized equation around singular minimizing cones. To be precise, suppose that $U$ is a minimizing  cone of \eqref{EqnAC}, the linearized equation around $U$ reads as \cite{CJK, DJS, JS}
\begin{equation}
\label{EqnLinearizedEquation}
\begin{cases}
\Delta\varphi=0 &\text{ in }\POSS,\\
\varphi_\nu+H\varphi=0 &\text{ on }\partial\POSS.
\end{cases}
\end{equation} 
Here $\nu$ is the inner unit normal of $\partial\POSS$, and $H$ denotes the mean curvature of $\partial\POSS$. 
Following the tradition in minimal surface theory, we refer to the linearized equation as the \textit{Jacobi equation} and solutions to the linearized equation as \textit{Jacobi fields}. 

The key insight behind generic regularity of minimal surfaces is a link between the decay of positive Jacobi fields and a cleaning estimate similar to \eqref{EqnNewCleaningIntro}. See, for instance,   \cite{CLS,CMS1,CMS2,CMSWa,LWa,Wa}, where this link is exploited using two classical ingredients: the bound on the principal eigenvalue for the Jacobi operator \cite{S,Z1} and the Harnack inequality on minimal surfaces \cite{BG}.   These ingredients are absent in the theory of free boundary problems. 

As part of the  main contributions of this work, we establish the following in the context of the Alt-Caffarelli problem \eqref{EqnAC}:
\begin{enumerate}
\item{a dimensional lower bound for the principal eigenvalue of the Jacobi equation on the sphere (Theorem~\ref{ThmDimensionalLowerBound}); and }
\item{a Harnack inequality for positive Jacobi fields (Theorem~\ref{ThmHarnack}).}
\end{enumerate}
These are of independent interest in the study of the Alt-Caffarelli problem and related problems, and constitute two important by-products of the analysis we perform here. 

Although similar results are well-known for minimal surfaces, new challenges arise in the context of free boundary problems. Instead of a single equation imposed on the minimal surface, here we face simultaneously two equations, one in the positive set $\PosS$ and one along the free boundary $\GU$. These two equations are in competition \cite{DJ, JS, SY}, leading to new difficulties. 

Below we explain some ideas behind the two ingredients above. 
%%%%%%%%%%%%%%%%%%%%%%%%%%%%%%%%%%%%%%%%%%%%%%%%%%%%%%%%%%%%%%%%%%%%%%%%%%%%%%%%%%%%%%%%%%%%%%%%%%%%%%%%%%%%%%%%%%%%%%%%%%%%%%%%%%%%%%%%%%%%%%%%%%%%%%%%%%%%%%%%%%%%%%%%%%%%%%%%%%%%%%%%%%%%%%%%%%%%%%%%%%%%%%%%%%%%%%%%%%%%%%
\subsection{Principal eigenvalue of the Jacobi equation on the sphere}\label{SubsectionPrincipalEigenValue}
At the infinitesimal scale, the separation between free boundaries is modeled by the decay rate of a positive Jacobi field on a minimizing cone, say, $U$ \cite{DJS, EdSV}.  After a separation of variables, this reduces to a lower bound on  the principal eigenvalue, $\lambda(U)$, in the following system
\begin{equation}
\label{EqnJacobiSph}
\begin{cases}
\Delta_{\Sph}\varphi=\lambda(U)\varphi &\text{ in }\{U>0\}^S,\\
\varphi_\nu+H\varphi=0 &\text{ on }\Gamma(U)^S,\\
\varphi>0 &\text{ in }\{U>0\}^S.
\end{cases}
\end{equation} 
Here and in the remaining part of this work, we denote by $E^S$ the trace of  a set $E$ on the sphere, that is,
\begin{equation}
\label{EqnES}
E^S:=E\cap\partial B_1,
\end{equation} 
and $\Delta_{\Sph}$ denotes the spherical Laplacian. 
 Recall that $\nu$ denotes the inner unit normal on the free boundary $\Gamma(U)$ and that $H$ denotes the mean curvature of the free boundary. 

For a minimizing cone $U$ with $\Sing(U)=\{0\}$, De Silva-Jerison-Shahgholian studied \eqref{EqnJacobiSph}  to quantify the rate at which nearby minimizers converge to $U$. Their estimate depends on the specific cone $U$ and crucially uses the smoothness of $\Gamma(U)^S$.

For our purpose, we show that, for any minimizing cone $U$ (possibly with nonempty $\Sing(U)^S$), we have
\begin{equation}
\label{EqnDimensionalBoundIntro}
\lambda(U)\ge \Lambda_d>0
\end{equation}
for a dimensional constant $\Lambda_d$. (See Theorem~\ref{ThmDimensionalLowerBound}.)

Even for cones with isolated singularity, this dimensional bound is completely new. The sharp value of $\Lambda_d$ remains an open question (see Remark~\ref{RemSharpGamma}).

\vem

For minimal surfaces, a similar estimate was shown for cones with an isolated singularity by Simons \cite{S} with the sharp value of $\Lambda_d$. The cones that attain equality in the estimate were classified by Perdomo \cite{P} and Wu \cite{Wu}. These were extended to general cones by Zhu \cite{Z1}. 

Among its applications, the sharp value of $\Lambda_d$ allowed Simons to rule out stable singular cones in seven dimensions \cite{S}. It  leads to the classification of entropy-stable cones by Zhu \cite{Z2}. The recent breakthrough by Chodosh-Mantoulidis-Schulze-Wang \cite{CMSWa} used crucially the full classification of cones with the extremal principal eigenvalue.

\vem

In the context of the Alt-Caffarelli problem, we lack tools like the Simons identity \cite{S, KaWa}, and less information seems to be available based on symmetry of the problem \cite{P}. Moreover,  due to the competition between the two equations in \eqref{EqnJacobiSph}, it is less clear what is the optimal test function to use \cite{JS}.  As a result,  despite the dimensional  bound in \eqref{EqnDimensionalBoundIntro}, the following important questions remain open, even for cones with isolated singularities:
\begin{op}
What is the sharp value of $\Lambda_d$? 
\end{op}
\begin{op}
In $\R^d$, what can be said about a minimizing cone $U$ with $\lambda(U)=\Lambda_d$?
\end{op}

\begin{rem}
\label{rem:numeric}
In the context of minimal surfaces, the optimal value of $\Lambda_d$ has been known for decades \cite{S, Z1}. This yields an explicit value of $\gamma_d$ as in \eqref{EqnNewCleaningIntro} that  is greater than 1 in all dimensions.
For the Alt-Caffarelli problem,  singular minimizing cones are much less understood. The only known example is  the axially symmetric cone  by De Silva-Jerison in $\R^7$ \cite{DJ}. 

For the axially symmetric cone in $\R^d$,  one can numerically compute the principal eigenvalue of the Jacobi operator on the sphere, $\lambda_d$. 
For $7\le d\le 14,$ these are given by
\[
\begin{array}{c@{\hspace{2em}}c@{\hspace{2em}}c@{\hspace{2em}}c}
\lambda_{7} \approx 5.70 & \lambda_{9} \approx 7.70 & \lambda_{11} \approx 9.70 & \lambda_{13} \approx 11.70 \\[6pt]
\lambda_{8} \approx 6.70 & \lambda_{10} \approx 8.70 & \lambda_{12} \approx 10.70 & \lambda_{14} \approx 12.70.
\end{array}
\]
The corresponding  values of $\gamma_d$ (see \eqref{EqnNewCleaningIntro} or \eqref{eq:gamma_below}) are given by:
\[
\begin{array}{c@{\hspace{2em}}c@{\hspace{2em}}c@{\hspace{2em}}c}
\gamma_{7} \approx 1.7573 & \gamma_{9} \approx 1.3672 & \gamma_{11} \approx 1.2523 & \gamma_{13} \approx 1.1934 \\[6pt]
\gamma_{8} \approx 1.4839 & \gamma_{10} \approx 1.2985 & \gamma_{12} \approx 1.2189 & \gamma_{14} \approx 1.1734.
\end{array}
\]

For the axially symmetric cone in $\R^d$ for $7\le d\le 14$, we observe that the values of $\lambda_d$ are approximately affine in $d$ and the value of $\gamma_d$ is greater than $1$. 
\end{rem}

%%%%%%%%%%%%%%%%%%%%%%%%%%%%%%%%%%%%%%%%%%%%%%%%%%%%%%%%%%%%%%%%%%%%%%%%%%%%%%%%%%%%%%%%%%%%%%%%%%%%%%%%%%%%%%%%%%%%%%%%%%%%%%%%%%%%%%%%%%%%%%%%%%%%%%%%%%%%%%%%%%%%%%%%%%%%%%%%%%%%%%%%%%%%%%%%%%%%%%%%%%%%%%%%%%%%%%%%%%%%%%
\subsection{Harnack inequality for the Jacobi equation}
With an integration-by-parts trick from Wang \cite{Wa}, the estimate \eqref{EqnDimensionalBoundIntro} leads to, in an integral form,  the rate of decay  of positive Jacobi fields on singular minimizing cones (see Lemma~\ref{LemWeakDecayJacobiField}). To iterate this estimate, we need to upgrade this integral estimate to a point-wise estimate (see  Proposition~\ref{PropDecayJacobiField}). The natural tool is a Harnack inequality for the Jacobi equation \eqref{EqnLinearizedEquation}.

In the context of minimal surfaces, this was established by Bombieri-Giusti \cite{BG}. For a general minimizer $u$ of \eqref{EqnAC}, a Harnack inequality was recently established by Edelen-Spolaor-Velichkov for harmonic functions in $\PosS$ with Neumann data on $\partial\PosS$ \cite{EdSV}.  Unfortunately, this  does not apply directly to the Jacobi equation due to the  different  boundary condition \eqref{EqnLinearizedEquation}.

For a minimizing cone $U$ with $\Sing(U)=\{0\}$, De Silva-Jerison-Shahgholian proved a Harnack inequality for \eqref{EqnLinearizedEquation}. For our purpose, we need a similar result  for general minimizing cones (possibly with nonempty $\Sing(U)^S$). This is a challenging task as we lack tools to study the concentration of Jacobi fields at singular points on the free boundary. In minimal surface theory, this was overcome by Cheeger-Naber \cite{CN}, Chodosh-Mantoulidis-Schulze \cite{CMS1} and Wang \cite{Wa} with the introduction of regularity scales, which quantifies the distance of a point to the singular set.

Inspired by this, we introduce the concept of regularity scales for the Alt-Caffarelli problem (see Definition~\ref{DefRegularityScales}). This allows us to extend the Harnack inequality by De Silva-Jerison-Shahgholian \cite{DJS} to general cones as in Theorem~\ref{ThmHarnack}.  A similar treatment works for other  equations posed in $\overline{\POSS}$, as long as  a Harnack inequality is available when the free boundary is smooth. 

\vem

This paper is organized as follows:

 In Section~\ref{sec:prelim}, we collect some preliminaries on the Alt-Caffarelli problem. In Section~\ref{SectionRegularityScales}, we introduce the concept of regularity scales and prove, for general minimizing cones, a Harnack inequality for \eqref{EqnLinearizedEquation}. In Section~\ref{sec:eigen}, we give the dimensional bound on the principal eigenvalue of the Jacobi equation  in \eqref{EqnDimensionalBoundIntro}. This allows us to quantify the rate of decay for positive Jacobi fields. In Section~\ref{sec:sep}, we use this information to study  the separation between singular sets of general minimizers. Finally in Section~\ref{sec:concl}, we establish the improved `cleaning' estimate \eqref{EqnNewCleaningIntro} as well as Theorem~\ref{ThmMainIntro} and Corollary~\ref{CorMainIntro}. 

%%%%%%%%%%%%%%%%%%%%%%%%%%%%%%%%%%%%%%%%%%%%%%%%%%%%%%%%%%%%%%%%%%%%%%%%%%%%%%%%%%%%%%%%%%%%%%%%%%%%%%%%%%%%%%%%%%%%%%%%%%%%%%%%%%%%%%%%%%%%%%%%%%%%%%%%%%%%%%%%%%%%%%%%%%%%%%%%%%%%%%%%%%%%%%%%%%%%%%%%%%%%%%%%%%%%%%%%%%%%%%%%%%%%%%%%%%%%%%%%%%%%%%%%%%%%%%%%%%%%%%%%%%%%%%%%%%%%%%%%%%%%%%%%%%%%%%%%%%%%%%%%%%%%%%%%%%%%%%%%%%%%%%%%%%%%%
\section*{Acknowledgements}
The authors thank Christos Mantoulidis and Zhihan Wang for fruitful discussions regarding this project. 

Part of this work was completed during the Thematic Programme on Free Boundary Problems at the Erwin Schr\"odinger International Institute for Mathematics and Physics, University of Vienna. The authors thank ESI for its hospitality and support during our visit.

%%%%%%%%%%%%%%%%%%%%%%%%%%%%%%%%%%%%%%%%%%%%%%%%%%%%%%%%%%%%%%%%%%%%%%%%%%%%%%%%%%%%%%%%%%%%%%%%%%%%%%%%%%%%%%%%%%%%%%%%%%%%%%%%%%%%%%%%%%%%%%%%%%%%%%%%%%%%%%%%%%%%%%%%%%%%%%%%%%%%%%%%%%%%%%%%%%%%%%%%%%%%%%%%%%%%%%%%%%%%%%%%%%%%%%%%%%%%%%%%%%%%%%%%%%%%%%%%%%%%%%%%%%%%%%%%%%%%%%%%%%%%%%%%%%%%%%%%%%%%%%%%%%%%%%%%%%%%%%%%%%%%%%%%%%%%%
\section{Preliminaries}
\label{sec:prelim}
For the reader's convenience, we collect here some preliminaries about the Alt-Caffarelli problem written in the form that will be used throughout the work. In the first subsection, we recall  properties of minimizers of the Alt-Caffarelli energy \eqref{EqnAC}.  In the second subsection, we turn to the Jacobi equation \eqref{EqnLinearizedEquation}. In the last subsection, we gather some  lemma from our previous work \cite{FeY} as well as some tools from \cite{FRS}.
%%%%%%%%%%%%%%%%%%%%%%%%%%%%%%%%%%%%%%%%%%%%%%%%%%%%%%%%%%%%%%%%%%%%%%%%%%%%%%%%%%%%%%%%%%%%%%%%%%%%%%%%%%%%%%%%%%%%%%%%%%%%%%%%%%
\subsection{Minimizers of the Alt-Caffarelli energy}
Recall the Alt-Caffarelli energy $\En(\cdot)$ from \eqref{EqnAC}. For a given (smooth) domain $\Omega\subset\R^d$ and a \textit{nonnegative} function $g\in H^{1}(\Omega)$, a \textit{minimizer in $\Omega$ with boundary data} $g$ is a function $u\in H^1(\Omega)$ satisfying $u=g$ on $\partial\Omega$ and 
$$
\En(u;\Omega)\le \En(v;\Omega) \hem\text{ for all }v=g \text{ on }\partial\Omega.
$$
In this case, we write
\begin{equation}
\label{EqnM}
u\in\M(\Omega;g).
\end{equation} 
 If, further, we have $p\in\GU$ with the free boundary $\GU$ defined in \eqref{EqnGU}, we write
\begin{equation}
\label{EqnMP}
u\in\M_p(\Omega;g).
\end{equation} 
We often omit the domain $\Omega$ or the boundary data $g$. 

With a slight abuse of notation, we write 
\begin{equation}
\label{EqnMRD}
u\in\M(\R^d)
\end{equation}
if $u$ is a \textit{global minimizer}, that is, if $u\in\M(B_R)$ for all $R>0.$

A special class of global minimizers consists of \textit{homogeneous minimizers} or \textit{minimizing cones}. The only minimizing cones with smooth free boundaries are rotations of the \textit{flat cone} 
\begin{equation}
\label{EqnFlatCone}
u_{\mathrm{flat}}(x)=(x_d)_+.
\end{equation}
Non-flat minimizing cones are referred to as \textit{singular minimizing cones}. 

In this work, the space of minimizing cones and the space of singular minimizing cones are denoted by $\mathcal{C}(\R^d)$ and $\mathcal{SC}(\R^d)$ respectively. That is, 
\begin{align}
\label{EqnCones}
&\mathcal{C}(\R^d):=\{U\in\M(\R^d):\hem U(tx)=t\cdot U(x) \text{ for all }t>0,x\in\R^d\}, \\
&\mathcal{SC}(\R^d):=\{U\in\mathcal{C}(\R^d):\hem U \text{ is not a rotation of} \hem\FlatC\}. \nonumber
\end{align} 

\vem

A minimizer satisfies the \textit{Euler-Lagrange equation}:
\begin{prop}[\cite{AC,V}]
\label{PropEL}
Suppose that $u\in\M$, then we have, in the viscosity sense,
$$
\Delta u=0 \hem\text{ in }\PosS, \text{ and }\hem |\nabla u|=1 \text{ on }\GU.
$$
\end{prop} 

Near a free boundary point, a minimizer is (quantitatively) Lipschitz and nondegenerate:
\begin{prop}[\cite{AC, V}]
\label{PropLipNonD}
For $u\in\M_0(B_1)$,  there are dimensional constants $c$ small and $C$ large such that 
$$
0 < cr\le\sup_{B_r}u\le Cr \hem\text{ for all }r\in(0,1/2),\hem
\text{ and }\hem |\nabla u|\le C \hem\text{in }B_{1/2}.$$
 \end{prop} 

\vem

The flat cone in \eqref{EqnFlatCone} is the model for minimizers with smooth free boundaries:
\begin{lem}[\cite{D1,V}]
\label{LemIOF}
Suppose that $u\in\M(B_1)$ satisifies
$$
|u-\FlatC|\le\eps \hem\text{ in }B_1.
$$
There are dimensional constants $\eps_d$ and  $\alpha$ small, and $C$ large such that if $\eps<\eps_d$, then $\PosS\cap B_{1/2}$ is a $C^{2,\alpha}$-epigraph in the $x_d$-direction, that is, we have
$$
\PosS\cap B_{1/2}=\{(x',x_d):\hem x_d>g(x')\}\cap B_{1/2},
$$
where $g:B_{1/2}\cap\{x_d=0\}\to\R$ satisfies
$$
\|g\|_{C^{2,\alpha}(B_{1/2}\cap\{x_d=0\})}\le C\eps.
$$

Moreover, we have
$$
\frac{\partial}{\partial x_d}u\ge 1/2 \hem\text{ in }\overline{\PosS}\cap B_{1/2}.
$$
\end{lem}

Lemma~\ref{LemIOF} gives a decomposition of the free boundary $\GU$:

\begin{defi}
\label{DefRegSing}
For $u\in\M$ and the constant $\eps_d$ from Lemma~\ref{LemIOF}, a point $p\in\GU$ is a \textit{regular point} if there is $r>0$ such that, up to a rotation, 
$$
|u-(x_d-p_d)_+|<\eps_d r \hem\text{ in }B_r(p).
$$
In this case, we write
$$
p\in\Reg(u).
$$
Otherwise, the point is a \textit{singular point} and we write
$$
p\in\SingU.
$$
\end{defi} 

%For the regular part, the initial regularity from Lemma~\ref{LemIOF} can be improved.
%\begin{thm}[\cite{AC,V}]
%\label{ThmAnalyticReg}
%For $u\in\M$, $\Reg(u)$ is locally an analytic hypersurface. 
%\end{thm} 
In particular, by Lemma~\ref{LemIOF},  along $\Reg(u)$ there is a well-defined unit normal, $\nu$, pointing towards $\PosS$. With equation (2.3) in \cite{JS}, the \textit{mean curvature} of the free boundary along $\Reg(u)$ can be computed as 
\begin{equation}
\label{EqnMeanCurvature}
H=-u_{\nu\nu}.
\end{equation} 
\vem

For information on the free boundary, we perform the blow-up analysis. To be precise, for $u\in\M_p(B_1)$  and $r>0$ small, the \textit{rescaled solution with center $p$  at scale $r$} is
\begin{equation}
\label{EqnRescaling}
u_{p,r}(x):=\frac{u(p+rx)}{r}.
\end{equation} 
Their limit as $r\to0$ captures the behavior of $u$ at the infinitesimal scale. To study this limit, one important tool is the  \textit{monotonicity formula} by Weiss \cite{W}
\begin{equation}
\label{EqnWeiss}
W(u;p,r):=\En(u_{p,r};B_1)-\int_{\partial B_1}u^2_{p,r},
\end{equation} 
where $\En(\cdot)$ is the Alt-Caffarelli energy from \eqref{EqnAC}.

This is a monotone quantity with respect to  $r$:
\begin{prop}[\cite{W}]
\label{PropWeiss}
For $u\in\M_p(B_1)$ (recall \eqref{EqnMP}), we have
\begin{equation}
\label{eq:eqhapp}
\frac{d}{dr}W(u;p,r)\ge 0 \hem\text{ for }r\in(0,1-|p|).
\end{equation}
In particular, the following quantity is well-defined
$$
W(u;p,0+):=\lim_{r\to0}W(u;p,r).
$$

If equality happens in \eqref{eq:eqhapp} at some $r>0$, then $u$ is homogeneous at $p$, that is,
$$
u(p+tx)=tu(p+x)\hem \text{ for all }t>0 \text{ and }x\in \R^d\hem \text{ with }\hem p+tx, p+x\in B_1.
$$
\end{prop}

In general, to study limits of minimizers, we have the following compactness result:
\begin{lem}[\cite{V}]
\label{LemCompactness}
Suppose that $\{u_n\}\subset\M(B_1)$ satisfies
$$
\Gamma(u_n)\cap B_{1/2}\neq\emptyset \hem\text{ for each }n.
$$
Then, up to a subsequence, we have
$$
u_n\to u_\infty \hem \text{ in }L^\infty_{loc}(B_1)\cap H^1_{loc}(B_1)
$$
for some $u_\infty\in\M(B_1)$.

Along the same subsequence, we have
$$
\overline{\{u_n>0\}}\to \overline{\{u>0\}} \hem\text{ and }\hem \Gamma(u_n)\to\Gamma(u_\infty)
$$
locally uniformly in the Hausdorff distance.

Moreover, if $p_n\in\Sing(u_n)\to p\in B_1,$ then $$p\in\Sing(u_\infty).$$
\end{lem} 

Applying this to the rescaled solutions gives the following:
\begin{prop}[\cite{V}]
\label{PropBlowUp}
Suppose that $u\in\M_p(B_1)$  and that $u_{p,r}$ is  from \eqref{EqnRescaling}, then along a subsequence of $r_n\to0$, we have 
$$
u_{p,r_n}\to U \hem\text{ in }L^\infty_{loc}(\R^d)\cap H^1_{loc}(\R^d),
$$
where $U\in\C$ satisfies
$$
W(U;0,R)=W(u;p,0+)=|\{U>0\}\cap B_1| \hem\text{ for all }R>0.
$$

Moreover, if $p\in\SingU$, then $U\in\mathcal{SC}(\R^d)$.
\end{prop} 
Recall the space of cones and singular cones, $\C$ and $\SC$, from \eqref{EqnCones}.
 
Together with  Proposition~\ref{PropLipNonD} and Proposition~\ref{PropWeiss}, the last equation in Proposition~\ref{PropBlowUp} implies
\begin{lem}
\label{LemWeissGap}
For $u\in\M_0(B_1)$ and $0<s<r<1/2$, we have
$$
0\le W(u;0,r)-W(u;0,s)\le C
$$
for a dimensional constant $C$. 
\end{lem} 

\vem

Proposition~\ref{PropBlowUp} reduces the study of $\SingU$ for a minimizer $u$ to the study of $\mathcal{SC}(\R^d)$ in \eqref{EqnCones}. In low dimensions, these are ruled out:
\begin{thm}[\cite{CJK, JS}]
\label{ThmNoSingularity}
For $d\le4,$ we have $\mathcal{SC}(\R^d)=\emptyset$.
\end{thm} 

 A dimension reduction argument  gives:
\begin{cor}[\cite{W}]
\label{CorEstSingular}
Suppose that $u\in\M(B_1)$ in $\R^d$ for $d\ge 5$, then we have 
$$
\DimH(\SingU)\le d-d^*-1,
$$
where $\DimH(\cdot)$ denotes the Hausdorff dimension, and $d^*\ge4$ is the critical dimension in \eqref{EqnCriticalDimension}.
\end{cor}

\vem

Compared with our previous result Theorem~\ref{ThmOldThm}, the main improvement behind Theorem~\ref{ThmMainIntro} follows from  the superlinear cleaning estimate \eqref{EqnNewCleaningIntro} on the separation of between minimizers. For simplicity, we introduce the following notation for pairs of ordered minimizers in $\M$ from \eqref{EqnM}.
\begin{defi}
\label{DefOrderedMinimizers}
For a domain $\Omega\subset\R^d$ and $u,v\in\M(\Omega)$, we say that $(u,v)$ is a pair of \textit{ordered minimizers} in $\Omega$ if
$$
u\le v \hem\text{ in }\Omega, \text{ and }\hem u<v \hem\text{ in }\PosS\cap\Omega.
$$ 
In this case, we write
$$
(u,v)\in\mathcal{OM}(\Omega).
$$
\end{defi} 

We get ordered minimizers from ordered boundary data:
\begin{prop}[Proposition 3.2 in \cite{FeY}]
\label{PropOrderingFromData}
Suppose that $u,v\in\M(B_1)$ with $v\ge u$ on $\partial B_1$. If in each connected component of $\PosS\cap\partial B_1$, there is a point $x_0$ such that $u(x_0)<v(x_0)$, then
$$
u\le v\hem \text{ in }B_1.
$$
\end{prop}

One fundamental result for ordered minimizers is the strict maximum principle (recall that $\Gamma(\cdot)$ denotes the free boundary \eqref{EqnGU}):
\begin{thm}[Corollary 1.2 in \cite{EdSV}]
\label{ThmStricMax}
Suppose that $(u,v)\in\OM(\Omega)$ and that $\{v>0\}$ is connected in $\Omega$, then 
$$
\GU\cap\Gamma(v)\cap\Omega=\emptyset.
$$
\end{thm}

Under a flatness assumption, the difference between ordered minimizers enjoys the Harnack inequality:
\begin{lem}[Proposition 5.1 in \cite{DJS}]
\label{LemHarnackForDifference}
Suppose that $(u,v)\in\OM(B_1)$ satisfies
$$
|u-(x_d)_+|+|v-(x_d)_+|<\eps \hem\text{ in }B_1.
$$

There are dimensional constants $\eps_d,\alpha$ small and $C$ large such that if $\eps<\eps_d$, then the difference $\varphi:=v-u$ satisfies
$$
\varphi(x)/\varphi(y)\le C \hem\text{ for all }x,y\in\overline{\PosS}\cap B_{1/2},
$$
and 
$$
\|\varphi\|_{C^{2,\alpha}(\overline{\PosS}\cap B_{1/2})}\le C.
$$
\end{lem} 

\vem

Global minimizers in $\M(\R^d)$ from \eqref{EqnMRD} are more rigid than minimizers in bounded domains. For instance, we have the following:
\begin{thm}[Theorem 2.3 in \cite{EdSV}]
\label{ThmGlobalConnected}
For $U\in\M(\R^d)$,  its positive set $\{U>0\}$ is connected.
\end{thm} 

We also have  the following quantitative estimates for global minimizers. Recall the mean curvature $H$ from \eqref{EqnMeanCurvature}.
\begin{lem}[Lemma 2.5 from \cite{EdSV}]
\label{LemGlobalQuant}
For $U\in\M(\R^d)$, we have
$$
|\nabla U|(p)\le 1 \hem\text{ for all }p\in\{U>0\},
\hem\text{ and } 
H(p)\ge0 \hem\text{ for all }p\in\Reg(U). 
$$

If equality is achieved in either inequality at a point, then, up to a rotation, we have
$U=\FlatC.$
\end{lem}

%%%%%%%%%%%%%%%%%%%%%%%%%%%%%%%%%%%%%%%%%%%%%%%%%%%%%%%%%%%%%%%%%%%%%%%%%%%%%%%%%%%%%%%%%%%%%%%%%%%%%%%%%%%%%%%%%%%%%%%%%%%%%%%%%%
\subsection{Jacobi equation around global minimizers}
\label{SubsectionJacobi}
For $U\in\M(\R^d)$ (see \eqref{EqnMRD}), the \textit{Jacobi equation around $U$} is given by
\begin{equation}
\label{EqnJacobi}
\begin{cases}
\Delta\varphi=0 &\text{ in }\{U>0\},\\
\varphi_\nu+H\varphi=0 &\text{ on }\Reg(U).
\end{cases}
\end{equation} 
Here $\nu$ is the inner normal on the regular part of the free boundary (see Definition~\ref{DefRegSing}), and $H$ is the mean curvature from \eqref{EqnMeanCurvature}. For derivation of this equation, we refer the reader to \cite{DJS,EdSV,JS}. 

The following is one of the main focuses of this work:
\begin{defi}
\label{DefPostiveJacobiField}
For $U\in\M(\R^d)$, a function $\varphi\in C^{2,\alpha}_{loc}(\overline{\POSS}\backslash\Sing(U)),
$ is a \textit{positive Jacobi field on }$U$ if 
$$\varphi>0\hem \text{ in }\overline{\POSS}\backslash\Sing(U),$$
and it satisfies the Jacobi equation around $U$ in  \eqref{EqnJacobi}.
\end{defi} 

For $U\in\SC$ from \eqref{EqnCones} with $\Sing(U)=\{0\}$, De Silva-Jerison-Shahgholian \cite{DJS} analyzed minimizers around $U$.  One key ingredient in their argument is the following estimate for positive Jacobi fields. 
\begin{lem}[Theorem 5.2 in \cite{DJS}]
\label{LemSchauder}
Suppose that $\varphi$ is a solution to
$$
\begin{cases}
\Delta\varphi=0 &\text{ in }B_1\cap\Omega,\\
\varphi_\nu+H\varphi=0 &\text{ on }B_1\cap\partial\Omega,
\end{cases}
$$
where the domain $\Omega\subset\R^d$ is of the form
$$
\Omega:=\{(x',x_d):\hem x_d>g(x')\}
$$for a $C^{2,\alpha}$-function $g$ with $g(0)=0$. 

Then we have the following:
\begin{enumerate}
\item{For a constant $C$ depending only on $d$ and $\|g\|_{C^{2,\alpha}}$, we have 
$$
\|\varphi\|_{C^{2,\alpha}(B_{1/2}\cap\overline{\Omega})}\le C\|\varphi\|_{L^\infty(B_1\cap\Omega)};
$$}
\item{For $\varphi>0$, there is a constant $C$ depending only on $d$ and $\|g\|_{C^{2,\alpha}}$ such that 
$$
\sup_{B_{1/2}\cap\overline{\Omega}}\varphi\le C\inf_{B_{1/2}\cap\overline{\Omega}}\varphi.
$$}
\end{enumerate}
\end{lem} 

\vem

Around $U\in\C$  from \eqref{EqnCones}, the Jacobi equation  leads to an eigenvalue problem on the sphere. 
The \textit{principal eigenvalue of the Jacobi operator around $U$} is the unique value $\lambda(U)$ for which the following system has a solution (recall the notation for spherical intersections $E^S$ from \eqref{EqnES})
\begin{equation}
\label{EqnSphJacobi}
\begin{cases}
\Delta_{\Sph}\varphi=\lambda(U) \varphi &\text{ in }\{U>0\}^S,\\
\varphi_\nu+H\varphi=0 &\text{ on }\Reg(U)^S,\\
\varphi>0 &\text{ in }\overline{\POSS}^S.
\end{cases}
\end{equation} 

This value $\lambda(U)$ is characterized by a variational problem \cite{JS}:
\begin{equation}
\label{EqnVariationalPbOnSph}
-\lambda(U)=\inf\{\mathcal{Q}_U(f):\hem f\in C_c^\infty(\partial B_1\backslash\Sing(U))\}.
\end{equation} 
Here  $\mathcal{Q}_U(f)$ is the quotient
\begin{equation}
\label{EqnQuotient}
\mathcal{Q}_U(f):=\frac{\int_{\POSS^S}|\nabla_\tau f|^2-\int_{\Reg(U)^S}Hf^2}{\int_{\POSS^S}f^2},
\end{equation}
where $\nabla_\tau$ denotes the tangential part of the gradient operator.

Stability gives the following bound on $\lambda(U)$ (see \cite[Proposition 2.1]{JS}):
\begin{equation}
\label{PropStabilityEigenBound}
\lambda(U)\le\frac{(d-2)^2}4\qquad\text{for}\quad U \in \C. 
\end{equation}
Consequently,  there are real roots to the  equation
\begin{equation}
\label{eq:gamma_below}
\gamma(\gamma-d+2)+\lambda(U)=0.
\end{equation}
Suppose that $\gamma$ is such a root and that $\varphi$ is a solution to \eqref{EqnSphJacobi}, then  we see that  $|x|^{-\gamma}\varphi(x/|x|)$ is a positive Jacobi field on $U$.

\vem

This links the decay of positive Jacobi fields to a lower bound on $\lambda(U)$. For our purpose, we need to bound the following
\begin{equation}
\label{EqnLambdaD}
-\Lambda_d:=\inf_{U\in\SC}\{\mathcal{Q}_U(f):\hem f\in C_c^\infty(\partial B_1\backslash\Sing(U))\}.  
\end{equation} 
Recall the space of singular cones $\SC$ from \eqref{EqnCones}.

By testing \eqref{EqnVariationalPbOnSph} with constants, it is not difficult to see that $\lambda(U)>0$ for each $U\in\SC$. However, even for cones with $\Sing(U)=\{0\}$, it is not known whether there is a dimensional bound on $\lambda(U)$. Such an estimate is at the heart of our argument. See Theorem~\ref{ThmDimensionalLowerBound}.

%%%%%%%%%%%%%%%%%%%%%%%%%%%%%%%%%%%%%%%%%%%%%%%%%%%%%%%%%%%%%%%%%%%%%%%%%%%%%%%%%%%%%%%%%%%%%%%%%%%%%%%%%%%%%%%%%%%%%%%%%%%%%%%%%
\subsection{Tools and previous results on generic regularity}
For many results on  generic regularity, the guiding principle is provided by the following tool from geometric measure theory:

\begin{lem}[Corollary 7.8 in \cite{FRS}]
\label{LemGenericReduction}
Suppose that $S$ is a subset of $\R^d\times(-1,1)$, and that $\pi_x$ and $\pi_t$ are the canonical projections 
$$
\pi_x(x,t)=x \hem\text{ and }\hem \pi_t(x,t)=t.
$$

Suppose that for some $p,m>0$, we have
\begin{enumerate}
\item{$\DimH(\pi_x(S))\le m$; and }
\item{For each $(x_0,t_0)\in S$ and $\eps>0$,  there exists $\rho=\rho(x_0,t_0,\eps)>0$ such that 
$$
S\cap\{(x,t)\in B_{\rho}(x_0)\times(-1,1):\hem t-t_0>|x-x_0|^{p-\eps}\}=\emptyset.
$$}
\end{enumerate}
Then we have the followings:
\begin{enumerate}
\item[(i)]{ if $m< p$, then
$$
\DimH(\pi_t(S))\le m/p;
$$}
\item[(ii)]{ if $m\ge p,$ then 
$$
\DimH(S\cap\pi_t^{-1}(t))\le m-p\hem \text{ for almost every } t\in(-1,1).
$$}
\end{enumerate}
\end{lem} 
Here $\DimH(\cdot)$ denotes the Hausdorff dimension. With \cite[Lemma 4.2]{FeR}, the case (i) holds for the Minkowski dimension.

With this, generic regularity relies on two ingredients, corresponding to the two assumptions on $S$ in Lemma~\ref{LemGenericReduction}. The main improvement in this work is for the second assumption (see \eqref{EqnNewCleaningIntro}). For the first assumption, we have already achieved the optimal bound. 

\begin{prop}[Propositions 4.6 and 4.7 in \cite{FeY}]
\label{PropOptimalResultFeY}
Suppose that $\FamPhi$ satisfies Assumption~\ref{Ass}. Denote the singular points in space-time as
$$
S:=\{(x,t)\in B_1\times(-1,1):\hem x\in\Sing(u_t) \text{ for some }u_t\in\M(B_1;g_t)\},
$$
then we have
\begin{enumerate}
\item{ if $d=d^*+1$, then
$$
\pi_t(S) \text{ is countable;}
$$}
\item{ if $d\ge d^*+2$, then 
$$
\DimH(\pi_x(S))\le d-d^*-1.
$$}
\end{enumerate}

\end{prop} 
\begin{rem}
In  \cite[Propositions 4.6 and 4.7]{FeY} the assumptions on the boundary data are stronger than Assumption~\ref{Ass}. Upon a quick inspection of the proofs, we see that  Assumption~\ref{Ass} is enough to bound the projection of the singular set, as stated here (up to replacing \cite[Corollary 4.4]{FeY} with Theorem~\ref{ThmStricMax} above). The stronger assumptions in \cite{FeY} are only used on the cleaning lemma there.
\end{rem}

%%%%%%%%%%%%%%%%%%%%%%%%%%%%%%%%%%%%%%%%%%%%%%%%%%%%%%%%%%%%%%%%%%%%%%%%%%%%%%%%%%%%%%%%%%%%%%%%%%%%%%%%%%%%%%%%%%%%%%%%%%%%%%%%%%%%%%%%%%%%%%%%%%%%%%%%%%%%%%%%%%%%%%%%%%%%%%%%%%%%%%%%%%%%%%%%%%%%%%%%%%%%%%%%%%%%%%%%%%%%%%
\section{Regularity scales and Harnack inequality for the Jacobi equation}
\label{SectionRegularityScales}
To obtain uniform estimates on the regular part of the free boundary, one obstruction is that regular points can converge to singular points. To overcome this,  we restrict to points with `quantified regularity'. This motivates the introduction of regularity scales, inspired by similar concepts in harmonic maps and minimal surfaces \cite{CN, CMS1,Wa}.

In the first subsection, we give the definition and basic properties of regularity scales. In the second subsection, we apply these  to establish the Harnack inequality for the Jacobi equation around a minimizing cone. This is one of the key technical contributions of this work. 

\subsection{Definition and basic properties of regularity scales}
For a given domain $\Omega\subset\R^d$, recall the space of minimizers $\M(\Omega)$ from \eqref{EqnM} and the definitions for  $\Reg(\cdot)$ and $\Sing(\cdot)$ from Definition~\ref{DefRegSing}. 

We introduce the regularity scales for a minimizer of the Alt-Caffarelli energy \eqref{EqnAC}:
\begin{defi}
\label{DefRegularityScales}
For  $u\in\M(\Omega)$, its \textit{regularity scale} in $\Omega$ is a function 
$$\rho_{u,\Omega}:\overline{\PosS}\cap\Omega\to[0,+\infty]$$
 given by
$$
\rho_{u,\Omega}(p):=\sup\{r>0: \hem B_r(p)\subset\Omega \hem\text{ and }\hem |D^2u|<r^{-1} \hem\text{ in }B_r(p)\cap\PosS\}.
$$
If the supremum is over an empty set, we define 
$$
\rho_{u,\Omega}(p)=0.
$$

For given $\rho\ge0$, the \textit{collection of points with regularity scale $\rho$} is denoted by $\mathcal{R}_{u,\Omega}(\rho)$, that is,
$$
\mathcal{R}_{u,\Omega}(\rho):=\{p\in\overline{\PosS}\cap\Omega:\hem \rho_{u,\Omega}(p)>\rho\}.
$$

For simplicity, we often omit the function $u$ or the domain $\Omega$ from the notations. 
\end{defi} 
The relation between the size of the ball and the bound on the Hessian is motivated by the following symmetry. We omit its elementary proof.
\begin{lem}
\label{LemScalingRS}
For $u\in\M(\Omega)$ and $r>0$, let the rescaled minimizer be defined as
$
u_r(x):=u(rx)/r.
$
Then we have
$$
\rho_{u_{r},\Omega/r}(p/r)= \frac{\rho_{u,\Omega}(p)}{r}\qquad\text{for}\quad p \in \overline{\{u > 0\}}\cap \Omega. 
$$
\end{lem} 

The following properties are direct consequences of Definition~\ref{DefRegularityScales}:
\begin{lem}
\label{LemBasicPropertiesOfRS}
Suppose that $u\in\M(\Omega)$, then we have:
\begin{enumerate}
 \item{the regularity scale $\rho_{u,\Omega}$ is  continuous  on $\overline{\PosS}\cap\Omega$;}
 \item{for each $\rho\ge0$, the collection $\mathcal{R}_{u,\Omega}(\rho)$ is open in $\overline{\PosS}\cap\Omega$; and}
 \item{for $0<\rho_1<\rho_2$, we have
 $$
 \mathcal{R}_{u,\Omega}(\rho_1)\supset \overline{\PosS}\cap\Omega\cap B_{\rho_2-\rho_1}(\mathcal{R}_{u,\Omega}(\rho_2)).
 $$}
\end{enumerate}
\end{lem} 
For the last item, we used the notation $B_\delta(E)$ for the $\delta$-neighborhood of a set $E$:
\begin{equation}
\label{EqnBDeltaE}
B_\delta(E):=\{p\in\R^d:\hem \mathrm{dist}(p,E)<\delta\}.
\end{equation} 

The collection $\{\RS(\rho)\}_{\rho>0}$ gives an open cover of $\PosS\cup\Reg(u)$:
\begin{prop}
\label{PropOpenCover}
For $u\in\M(\Omega)$, we have
$$
\PosS\cup\Reg(u)\cap\Omega\subset\bigcup_{\rho>0}\RS(\rho).
$$
\end{prop} 
\begin{proof}
For a point $p\in\PosS\cap\Omega$, we find $r>0$ such that $B_r(p)\subset\PosS\cap\Omega$. With $u$ being a positive harmonic function in $B_r(p)$, we get a bound on its Hessian in a neighborhood of $p$. This gives  $p\in\RS(\rho)$ for some $\rho>0$. 

For a point $p\in\Reg(u)\cap\Omega$, by Lemma~\ref{LemIOF}, we find $r>0$ such that $B_r(p)\subset\Omega$ and that $\PosS\cap B_r(p)$ is the epigraph of a smooth function. In this domain, we apply estimates for harmonic functions to get a bound on $|D^2u|$. This implies $p\in\RS(\rho)$ for some $\rho>0.$
\end{proof} 

The inclusion in Proposition~\ref{PropOpenCover} is an equality. This is a consequence of Lemma~\ref{LemIOF} and the following:
\begin{lem}
\label{LemFollowingOpenCover}
Suppose that $u\in\M_0(B_1)$ with $0\in\RS(\rho)$ for some $\rho>0$.

There is a dimensional constant $\mu_d>0$ such that for given $\eps>0$, we have, up to a rotation, 
$$
|u-(x_d)_+|<\eps  r \text{ in }B_r
$$
for all $r<\mu_d\eps\rho$.
\end{lem} 
Recall the set of minimizers $\M_p(\cdot)$ from \eqref{EqnMP}.

\begin{proof}
Up to a rotation,  Proposition~\ref{PropEL} gives
$
\nabla u(0)=e_d.
$
With $0\in\RS(\rho)$,we have
$$
|u-x_d|<C_dr^2/\rho \hem\text{ in }B_r\cap\PosS
$$
for $r<\rho$ and a dimensional constant $C_d$. The conclusion follows by choosing $\mu_d<1/C_d$.
\end{proof} 

It follows from Proposition~\ref{PropOpenCover} that $\RS(\rho)$ is nonempty for small $\rho$. This can be quantified:
\begin{lem}
\label{LemAnchoringPoint}
Suppose that $u\in\M_0(B_2)$, then, up to a rotation, we have
$$
B_{r_d}(e_1)\subset\PosS\cap\RS(\rho_d)
$$
for dimensional constants $r_d$ and $\rho_d$.
\end{lem} 

\begin{proof}
Proposition~\ref{PropLipNonD} gives, up to a rotation, that 
$
u(e_1)>c_d>0.
$
 The same proposition gives dimensional constants $r_d<1/4$ and $C_d$ such that 
$$
0<u<C_d\hem \text{ in }B_{2r_d}(e_1).
$$ 
The conclusion follows from the Hessian bound for harmonic functions. 
\end{proof} 

The following provides the usefulness of regularity scales. For points with a lower bound on their regularity scales, their limit points satisfy the same bound. 
\begin{thm}
\label{ThmContinuityOfRS}
For each $n\in\N$, suppose that $u_n\in\M(B_1)$ and $p_n\in\overline{\{u_n>0\}}\cap B_1$. 

If we have $$
u_n\to u \hem\text{ in }L^\infty_{loc}(B_1)\hem\text{ and }\hem p_n\to p\in B_1\quad\text{as}\quad n \to \infty,
$$
then 
$$
\lim_{n\to \infty}\rho_{u_n}(p_n)=\rho_{u}(p).
$$
\end{thm} 

With Lemma~\ref{LemCompactness}, we see that $u\in\M(B_1)$ and $p\in\overline{\PosS}$.

Theorem~\ref{ThmContinuityOfRS} follows from Lemma~\ref{LemLimSup} and Lemma~\ref{LemLimInf} below, corresponding to the $\liminf$ and $\limsup$ inequalities, respectively. 

\begin{lem}
\label{LemLimSup}
Under the assumptions   in Theorem~\ref{ThmContinuityOfRS}, suppose that, for some $\rho>0$,
$$
p_n\in\RS_{u_n}(\rho)  \hem\text{ for all }n,
$$
then 
$$
p\in\RS_u(\rho') \hem\text{ for all }\hem 0<\rho'<\rho.
$$
\end{lem} 
\begin{proof}
For $\rho'<\rho$, with $p_n\to p$ and  $B_{\rho}(p_n)\subset B_1$ for all $n$, we see that
$
B_{\rho'}(p)\subset B_1.
$
It remains to bound $|D^2u(q)|$ by $1/\rho'$ for $q\in\PosS\cap B_{\rho'}(p).$

With $q\in\PosS\cap B_1$, we find $0<\delta<(\rho-\rho')/4$ such that 
$$
u>\delta \hem\text{ in }B_{2\delta}(q)\subset B_1.
$$
Uniform convergence of $u_n$ to $u$ gives $u_n>0$ in $B_{2\delta}(q)$ for  large $n$. It follows from Proposition~\ref{PropEL} that $(u_n-u)$ is a harmonic function in $B_{2\delta}(q)$.
Hessian bound on harmonic functions gives
$$
|D^2u(q)-D^2u_n(q)|\le C\delta^{-2}\|u_n-u\|_{L^\infty(B_{2\delta}(q))}<\frac{1}{\rho'}-\frac{1}{\rho}
$$
for large $n$.

On the other hand, with $q\in B_{\rho'}(p)$ and $p_n\to p$, we have
$
|q-p_n|<\rho
$
for large $n$. With $p_n\in\RS_{u_n}(\rho)$, this gives $|D^2u(q)|<1/\rho$. Combining this with the previous inequality, we get
$
|D^2u(q)|<1/\rho'.
$
\end{proof} 

\begin{lem}
\label{LemLimInf}
Under the assumptions in Theorem~\ref{ThmContinuityOfRS}, suppose that
$$
p\in\RS_{u}(\rho) \hem\text{ for some }\rho>0.
$$
Then, for $0<\rho'<\rho$, we have
$$
p_n\in\RS_{u_n}(\rho') \hem\text{ for  large }n.
$$
\end{lem} 

\begin{proof}
Suppose the conclusion fails, then by taking a subsequence, we find 
\begin{equation}
\label{EqnQNContra}
q_n\in B_{\rho'}(p_n)\cap\{u_n>0\}\hem\text{ and }\hem |D^2u_n(q_n)|\ge\frac{1}{\rho'} 
\end{equation}
for all $n\in\N$ large.  With $p_n\to p$, up to taking a further subsequence, we have
$$
q_n\to q\in\overline{B_{\rho'}(p)}\cap\overline{\PosS}.
$$

If $q\in\PosS$, the same argument as in the proof of Lemma~\ref{LemLimSup} gives $|D^2u_n(q_n)|<1/\rho'$ for large $n$, leading to a contradiction with \eqref{EqnQNContra}. As a result,  it suffices to consider the case where
$$
q\in\Gamma(u).
$$

\vem

With $p\in\RS_{u}(\rho)$ and $q\in\overline{B_{\rho'}(p)}$, we see that 
$$
|D^2u|<\frac{1}{\rho} \hem\text{ in }B_{\rho-\rho'}(q)\cap\PosS.
$$
Thus $q\in\RS_u(\rho-\rho')$.

Up to a rotation, we apply Lemma~\ref{LemFollowingOpenCover} to find $r>0$ such that 
$$
|u-(x_d-q_d)_+|<\frac14\eps_d r \hem\text{ in }B_r(q).
$$
Locally uniform convergence of $u_n$ to $u$ gives
$$
|u_n-(x_d-q_d)_+|<\frac12\eps_d r \hem\text{ in }B_r(q)
$$
for all large $n$. 

With Lemma~\ref{LemIOF}, we see that $\PosS$ and $\{u_n>0\}$ are $C^{2,\alpha}$-epigraphs in $B_{r/2}(q)$ with uniform $C^{2,\alpha}$-norm, depending only on $r$ and $d$. Moreover, we have
$$
\frac{\partial}{\partial x_d}u\ge1/2 \hem\text{ in }B_{r/2}(q)\cap\PosS.
$$

For $\delta>0$ small to be chosen, take
$$
\bar{q}=q+\delta e_d.
$$
Then 
$
u(\bar{q})\ge\frac12\delta.
$ 
Proposition~\ref{PropLipNonD} gives a dimensional constant $c>0$ such that 
$$
B_{c\delta}(\bar{q})\subset\PosS\cap\{u_n>0\} \hem\text{ for large }n.
$$
Estimates for the harmonic function $(u_n-u)$ in this domain give
$$
|D^2u_n(\bar{q})-D^2u(\bar{q})|\le C\delta^{-2}\|u_n-u\|_{L^\infty(B_{\delta}(\bar{q}))}
$$
for all large $n$, where $C$ is a dimensional constant. 

Estimates for the harmonic function $u_n$ in $\{u_n>0\}\cap B_{r/2}(q)$ give
$$
|D^2u_n(q_n)|\le |D^2u_n(q_n)-D^2u_n(\bar{q})|\le C_r[\delta+|q_n-q|].
$$
Combining with the previous estimate, we have
$$
|D^2u_n(q_n)|\le |D^2u(\bar{q})|+C_r[\delta+|q_n-q|]+C\delta^{-2}\|u_n-u\|_{L^\infty(B_{\delta}(\bar{q}))}.
$$

Finally, with $|\bar{q}-q|=\delta$ and $q\in \overline{B_{\rho'}(p)}$, we see that $\bar{q}\in B_{\rho}(p)$ if $\delta$ is small. With $p\in\RS_u(\rho)$, we get $|D^2u(\bar{q})|<1/\rho$. Putting this into the previous estimate, we have
$$
|D^2u_n(q_n)|\le\frac{1}{\rho}+C_r[\delta+|q_n-q|]+C\delta^{-2}\|u_n-u\|_{L^\infty(B_{\delta}(\bar{q}))}.
$$
We pick $\delta$ small such that $C_r\delta<\frac14[\frac{1}{\rho'}-\frac{1}{\rho}].$ Then 
$$
|D^2u_n(q_n)|\le\frac{1}{\rho'}+\frac34\left[\frac{1}{\rho}-\frac{1}{\rho'}\right]+C_r|q_n-q|+C\delta^{-2}\|u_n-u\|_{L^\infty(B_{\delta}(\bar{q}))}.
$$
With $q_n\to q$ and $u_n\to u$ locally uniformly, this contradicts \eqref{EqnQNContra} for large $n$. 
\end{proof}

%%%%%%%%%%%%%%%%%%%%%%%%%%%%%%%%%%%%%%%%%%%%%%%%%%%%%%%%%%%%%%%%%%%%%%%%%%%%%%%%%%%%%%%%%%%%%%%%%%%%%%%%%%%%%%%%%%%%%%%%%%%%%%%%%%%%%%%%%%%%%%%%%%%%%%%%%%%%%%%%%%%%%%%%%%%%%%%%%%%%%%%%%%%%%%%%%%%%%%%%%%%%%%%%%%%%%%
\subsection{Harnack inequality for the Jacobi equation}
Based on Lemma~\ref{LemSchauder}, De Silva-Jerison-Shahgholian analyzed positive Jacobi fields on cones with smooth sections on the sphere. One of the key technical contribution here is to extend part of their analysis to more general cones. 

The starting point is the following lemma on the connectivity of $\RS(\rho)$ from Definition~\ref{DefRegularityScales}. Recall also the notation for minimizing cones $\C$ from \eqref{EqnCones}.
\begin{lem}
\label{LemConnectivity}
For $U\in\C$, let $r_d,\rho_d$ be constants from Lemma~\ref{LemAnchoringPoint} such that, up to a rotation, 
\begin{equation}
\label{EqnAnchoringPointAssumption}
B_{r_d}(e_1)\subset\{U>0\}\cap\RS_{U}(\rho_d).
\end{equation}
Given $\rho>0$, there is $\delta\in(0,1/4)$, depending only on $\rho$ and  $d$, such that the following holds:

For each $p\in\RS(\rho)\cap\partial B_1$, there is a continuous curve
$$
\gamma:[0,1]\to\overline{\POSS}\cap\partial B_1
$$
satisfying
$$
\gamma(0)=p,\hem \gamma(1)=e_1,\hem 
\gamma([0,1/2])\subset B_{\mu_d\eps_d\rho/8}(p),
$$
and
$$
B_\delta(\gamma([1/2,1]))\subset\POSS,
$$
where $\mu_d$ is  from Lemma~\ref{LemFollowingOpenCover} and $\eps_d$ is from Lemma~\ref{LemIOF}.
\end{lem} 
Recall our notation for neighborhoods of sets $B_\delta(E)$ from \eqref{EqnBDeltaE}.

\begin{proof}
For given $\rho>0$, suppose that there is no such $\delta$, we find  a sequence $U_n\in\C$ with \eqref{EqnAnchoringPointAssumption} and a sequence $p_n\in\RS_{U_n}(\rho)\cap\partial B_1$ such that whenever $\gamma$ is a continuous map into $\overline{\{U_n>0\}}\cap\partial B_1$ with $\gamma(0)=p_n$, $\gamma(1)=e_1$, we must have 
\begin{equation}
\label{EqnHemHemContra}
\text{ either }\hem
\gamma([0,1/2])\not\subset B_{\mu_d\eps_d\rho/8}(p_n)
\hem\text{ or }\hem
B_{1/n}(\gamma([1/2,1])\not\subset\{U_n>0\}.
\end{equation}
Up to a subsequence, Lemma~\ref{LemCompactness} and Theorem~\ref{ThmContinuityOfRS} gives $U\in\C$ such that 
$$
U_n\to U \hem\text{ in }L_{loc}^{\infty}(\R^d)
\hem\text{ and }\hem
p_n\to p\in\partial B_1\cap\RS_U(\rho/2).
$$
Depending on whether $p$ lies in $\POSS$ or $\Gamma(U)$, we consider two cases.

\vem

\textit{Case 1: $p\in\POSS.$}

In this case, we find $0<\eps<\mu_d\eps_d\rho/20$ such that 
$
U>\eps \text{ in }B_{2\eps}(p).
$
For large $n$, uniform convergence of $U_n$ to $U$ gives
$$
U_n>0 \hem\text{ in }B_{2\eps}(p).
$$

With $p_n\to p$, we have $p_n\in B_\eps(p)$ for large $n$. Take a continuous curve $\gamma_n:[0,1/2]\to\partial B_1\cap B_\eps(p)$ with $\gamma_n(0)=p_n$ and $\gamma_n(1/2)=p$, then
\begin{equation}
\label{EqnFirstEquationFirstCase}
\gamma_n([0,1/2])\subset B_{\mu_d\eps_d\rho/20}(p_n) \hem\text{ and }\hem B_\eps(\gamma_n([0,1/2]))\subset\{U_n>0\}
\end{equation}
for large $n$.

With Theorem~\ref{ThmGlobalConnected} and the homogeneity of $U\in\C$, we find a continuous curve $\gamma_n:[1/2,1]\to\partial B_1\cap\{U>0\}$ such that $\gamma_n(1/2)=p$ and $\gamma_n(1)=e_1$. Compactness of $[1/2,1]$ implies
$$
m:=\min_{\gamma_n([1/2,1])}U>0.
$$
Uniform convergence of $U_n$ to $U$ together with Proposition~\ref{PropLipNonD} gives 
\begin{equation}
\label{EqnSecondEquationInFirstCase}
B_{c_dm}(\gamma_n([1/2,1]))\subset\{U_n>0\}
\end{equation}
for large $n$, where $c_d$ is a dimensional constant. 

Joining the curve $\gamma_n$ from $[0,1/2]$ and $[1/2,1]$, we get a continuous map  into $\{U_n>0\}\cap\partial B_1$ with $\gamma_n(0)=p_n$ and $\gamma_n(1)=e_1$. With \eqref{EqnFirstEquationFirstCase} and \eqref{EqnSecondEquationInFirstCase}, we have a contradiction to \eqref{EqnHemHemContra}.

\vem

\textit{Case 2: $p\in\Gamma(U)$.}

Take $r:=\mu_d\eps_d\rho/8$, then Lemma~\ref{LemFollowingOpenCover} gives, up to a rotation,
$$
|U-(x_d-p_d)_+|<\frac12\eps_d r\hem\text{ in }B_r(p)
$$
and
$$
|U_n-(x_d-p_d)_+|<\eps_d r\hem\text{ in }B_r(p)
$$
by convergence of $U_n$ to $U$.

Consequently, Lemma~\ref{LemIOF} implies that $\POSS$ and $\{U_n>0\}$ are $C^{2,\alpha}$-epigraphs in $B_{r/2}(p)$. Moreover,  for $\bar{p}:=\frac{p+\frac14re_d}{|p+\frac14re_d|}$, we have
$
U(\bar{p})\ge r/8
$
and
$$
\bar{p}\in\partial B_1\cap B_{r/2}(p)\cap\{U_n>0\}
$$
for large $n$. 

With $\{U_n>0\}$ being an epigraph in $B_{r/2}(p)$ and $p_n\to p$, we find a continuous map 
$$
\gamma_n:[0,1/2]\to\partial B_1\cap B_{r/2}(p)\cap\overline{\{U_n>0\}}
$$
such that 
$$
\gamma_n(0)=p_n \hem\text{ and }\hem\gamma_n(1/2)=\bar{p}.
$$
Our choice of $r$ implies
\begin{equation}
\label{EqnFirstEquationSecondCase}
\gamma_n([0,1/2])\subset B_{\mu_d\eps_d\rho/8}(p_n)
\end{equation} 
for large $n$. 

With similar argument as in \textit{Case 1}, we find $\gamma_n:[1/2,1]\to\partial B_1\cap\{U_n>0\}$ with $\gamma_n(1/2)=\bar{p}$, $\gamma_n(1)=e_1$ and 
\begin{equation}
\label{EqSecondEquationSecondCase}
B_\eta(\gamma_n([1/2,1]))\subset\{U_n>0\})
\end{equation} 
for some $\eta>0$ independent of $n$. 

Joining $\gamma_n$ from the two sub-intervals, we have $\gamma_n(0)=p_n$ and  $\gamma_n(1)=e_1$. With \eqref{EqnFirstEquationSecondCase} and \eqref{EqSecondEquationSecondCase}, we contradict \eqref{EqnHemHemContra} for large $n$. 
\end{proof} 

With this preparation, we give the main result of this section:
\begin{thm}
\label{ThmHarnack}
Suppose that $\varphi$ is a positive Jacobi field on $U\in\C$. 

For $\rho>0$, there is a constant $C$, depending only on $d$ and $\rho$, such that 
$$
\sup_{\partial B_1\cap\RS_{U}(\rho)}\varphi\le C\inf_{\partial B_1\cap\RS_{U}(\rho)}\varphi.
$$
\end{thm} 
Recall the notion of positive Jacobi fields from Definition~\ref{DefPostiveJacobiField} and the space of minimizing cones $\C$ from \eqref{EqnCones}.

\begin{proof}
Up to a rotation, Lemma~\ref{LemAnchoringPoint} gives
$$
B_{r_d}(e_1)\subset\POSS\cap\RS(\rho_d)
$$
for dimensional $r_d,\rho_d>0$.

Given $\rho>0$, we pick an arbitrary $p\in\RS(\rho)$. It suffices to show that $\varphi(p)/\varphi(e_1)$ is bounded away from $0$ and infinity.

For this $\rho$, let $\delta$ be the constant from Lemma~\ref{LemConnectivity}. Let $\gamma$ be the curve from the same lemma connecting $p$ to $e_1$.

We bound $\varphi(p)/\varphi(e_1)$  in two steps. In the first step, we bound $\varphi(\gamma(1/2))/\varphi(e_1)$. In the second, we bound $\varphi(p)/\varphi(\gamma(1/2))$.

\vem

\textit{Step 1: Comparing $\varphi(\gamma(\frac12))$ with $\varphi(e_1)$.}

Define $t_0=1/2$ and $p_0=\gamma(t_0)$. 

Once $\{(t_j,p_j)\}_{j=0,1,2,\dots,n}$ have been picked, we take $p_{n+1}:=\gamma(t_{n+1})$ with
\begin{equation}
\label{TN+1}
t_{n+1}:=\inf\{t\in[1/2,1]:\hem B_{\delta/4}(\gamma(t))\cap[{\textstyle\bigcup_{j=0,1,2,\dots,n}}B_{\delta/4}(p_j)]\}.
\end{equation}
If the set on the right-hand side is empty, we terminate the process. 

By construction, we have
$$
\overline{B_{\delta/4}(p_{n+1})}\cap[{\textstyle \bigcup_{j=0,1,2,\dots,n}}\overline{B_{\delta/4}(p_j)}]\neq\emptyset.
$$
That is,  for some $j\le n$, 
$$
p_{n+1}\in \overline{B_{\delta/2}(p_j)}\subset B_{\delta}(p_j)\subset\POSS,
$$
where the last inclusion is from Lemma~\ref{LemConnectivity}.

With $\varphi$ being a positive harmonic function in $\POSS$, the Harnack inequality gives 
$$
c_d \varphi(p_j)\le \varphi(p_{n+1})\le C_d\varphi(p_j)
$$
for dimensional constants $c_d$ and $C_d$. 
Iterate this process, we get 
\begin{equation}
\label{LALALA}
c_d^n \varphi(p_0)\le \varphi(p_{n+1})\le C_d^n\varphi(p_0).
\end{equation}

It follows from the construction \eqref{TN+1} that $\{B_{\delta/4}(p_j)\}_{j=0,1,2,\dots,n+1}$ is a family of disjoint balls in $B_3$. Thus this process has to terminate in $N$ steps, with $N\le C\delta^{-d}$ for dimensional  $C$. It follows from \eqref{LALALA} that, for any $j=1,2,\dots, N$, we have
\begin{equation}
\label{LA}
c\varphi(p_0)\le\varphi(p_j)\le C\varphi(p_0)
\end{equation}
for constants $c$ and $C$ depending only on $d$ and $\delta.$

With the process terminating at step $N$ and that $e_1=\gamma(1)$, we see that 
$$
B_{\delta/4}(e_1)\cap[{\textstyle \bigcup_{j=0,1,2,\dots,N}}B_{\delta/4}(p_j)]\neq\emptyset.
$$
The same argument leading to \eqref{LALALA} gives
$$
c\varphi(p_j)\le\varphi(e_1)\le C\varphi(p_j)
$$
for some $j=1,2,\dots,N.$ Combined with \eqref{LA}, we get
$$
c\varphi(\gamma(1/2))\le \varphi(e_1)\le C\varphi(\gamma(1/2))
$$
for constants $c$ and $C$ depending only on $d$ and $\rho$.

\vem

\textit{Step 2: Comparing $\varphi(p)$ with $\varphi(\gamma(1/2))$.}

Recall from Lemma~\ref{LemConnectivity} that 
$$
\gamma(1/2)\in B_{\mu_d\eps_d\rho/8}(p).
$$
As a result, if $B_{\mu_d\eps_d\rho/4}(p)\subset\POSS$, we can apply  Harnack inequality for the harmonic function $\varphi$ to conclude
$$
c\varphi(\gamma(1/2))\le\varphi(p)\le C\varphi(\gamma(1/2))
$$
for dimensional constants $c$ and $C$. 

It remains to consider the case when we can find 
$$
q\in B_{\mu_d\eps_d\rho/4}(p)\cap\Gamma(U).
$$
With Lemma~\ref{LemIOF} and Lemma~\ref{LemFollowingOpenCover}, we see that $\POSS\cap B_{\mu_d\eps_d\rho/2}(q)$ satisfies the assumption in Lemma~\ref{LemSchauder}. Note that $\gamma(1/2)\in B_{\mu_d\eps_d\rho/8}(p)\subset B_{3\mu_d\eps_d\rho/8}(q)$, Lemma~\ref{LemSchauder} gives
$$
c\varphi(\gamma(1/2))\le\varphi(p)\le C\varphi(\gamma(1/2))
$$
for constants $c$ and $C$ depending only on $d$ and $\rho$. 
\end{proof}

%%%%%%%%%%%%%%%%%%%%%%%%%%%%%%%%%%%%%%%%%%%%%%%%%%%%%%%%%%%%%%%%%%%%%%%%%%%%%%%%%%%%%%%%%%%%%%%%%%%%%%%%%%%%%%%%%%%%%%%%%%%%%%%%%%%%%%%%%%%%%%%%%%%%%%%%%%%%%%%%%%%%%%%%%%%%%%%%%%%%%%%%%%%%%%%%%%%%%%%%%%%%%%%%%%%%%%%%%%%%%%
\section{Principal eigenvalue of the Jacobi equation}
\label{sec:eigen}
Based on  Subsection~\ref{SubsectionJacobi}, the decay of positive Jacobi fields is determined by the principal eigenvalue of the Jacobi operator on the sphere. In the first subsection below, we give a dimensional lower bound on this eigenvalue. Even for minimizing cones with smooth sections on the sphere, this lower bound is new. In the second subsection, we use this to quantify the decay rate of positive Jacobi fields. The techniques here are inspired by the works of Simon \cite{Si}, Wang \cite{Wa}, and Zhu \cite{Z1}.

%%%%%%%%%%%%%%%%%%%%%%%%%%%%%%%%%%%%%%%%%%%%%%%%%%%%%%%%%%%%%%%%%%%%%%%%%%%%%%%%%%%%%%%%%%%%%%%%%%%%%%%%%%%%%%%%%%%%%%%%%%%%%%%%%%%%%%%%%%%%%%%%%%%%
\subsection{Lower bound on the principal eigenvalue}
For a minimizing cone $U$, Lemma~\ref{LemGlobalQuant} gives a sign on  the mean curvature $H$ from \eqref{EqnMeanCurvature}. Moreover, this quantity vanishes at a point only for the flat cone in \eqref{EqnFlatCone}. 
We start with the stability of this classification:
%\begin{lem}
%\label{LemStabilityGradientGap}
%For $U\in\SC$, there is a dimensional constant $\delta_d>0$ such that 
%$$
%\int_{\POSS^S}(1-|\nabla U|^2)\ge\delta_d.
%$$
%\end{lem} 
%Recall the space of singular minimizing cones $\SC$ from \eqref{EqnCones}. Recall $E^S$, our notation for the spherical intersection of a set $E$, from \eqref{EqnES}.

%\begin{proof}
%Suppose that there is no such $\delta_d>0$,  we find a sequence $U_n\in\SC$ with
%\begin{equation}
%\label{TempEqn}
%\int_{B_1\cap\{U_n>0\}}(1-|\nabla U_n|^2)\to0.
%\end{equation}
%Lemma~\ref{LemCompactness} implies that, up to a subsequence, $U_n$ converges to $U\in\SC$ in $H^1(B_1)$. Lemma~\ref{LemGlobalQuant} gives
%$$
%\int_{B_1\cap\POSS}(1-|\nabla U|^2)>0.
%$$
%This contradicts \eqref{TempEqn} for large $n$. 
%\end{proof} 

\begin{lem}
\label{LemLowerBoundOnHessian}
For $U\in\SC$, there is a dimensional constant $\delta_d>0$ such that 
$$
\int_{\POSS^S}|D^2U|^2\ge\delta_d.
$$
\end{lem}
Recall the space of singular minimizing cones $\SC$ from \eqref{EqnCones}, and the notation for spherical intersections $E^S$ from \eqref{EqnES}.

\begin{proof}
Suppose not, we find a sequence $U_n\in\SC$ such that 
\begin{equation}
\label{JusForCotra}
\int_{B_2\cap\{U_n>0\}}|D^2U_n|^2\to0.
\end{equation} 
With Lemma~\ref{LemAnchoringPoint}, we can assume that $B_{r_d}(e_1)\subset\{U_n>0\}$ for a dimensional constant $r_d>0$.

Lemma~\ref{LemCompactness} implies that, up to a subsequence, we have
$$
U_n\to U\in\SC \hem\text{ locally uniformly in }\R^d \ \text{ and in }C^2(B_{r_d/2}(e_1)).
$$
With \eqref{JusForCotra}, we have 
$$
|D^2U|^2(e_1)=0.
$$

Since $U$ is harmonic in $\POSS$, $|D^2U|^2$ is subharmonic in the same set ($\Delta |D^2 U|^2 = 2|D^3 u|^2\ge 0$). With the connectedness of $\POSS$ from Theorem~\ref{ThmGlobalConnected}, the strong maximum principle implies that $D^2U=0$ in $\POSS$. Up to a rotation, this forces $U=\FlatC$, contradicting $U\in\SC.$
\end{proof} 

The $W^{2,2}$ norm of $U$ is controlled in terms of the mean curvature:
\begin{lem}
\label{LemMeanCurvatureControlsHessian}
For $U\in\C$, we have
$$
\int_{\POSS^S}|D^2U|^2\le\int_{\Reg(U)^S}H.
$$ 
\end{lem} 

\begin{proof}
For positive $\eps$ and $\delta$, Corollary~\ref{CorEstSingular} and the compactness of $\Sing(U)^S$ give a finite collection $\{B_{r_j}(x_j)\}_{j=1,2,\dots,N}$ satisfying
$$
x_j\in\Sing(U)^S,\hem  0<r_j<\delta, \hem\cup B_{r_j}(x_j)\supset\Sing(U)^S,\hem \text{ and }\hem
\sum r_j^{d-3}<\eps.
$$
For each $B_{r_j}(x_j)$, we find a smooth function $\varphi_j$ on $\partial B_1$ with values in $[0,1]$ that satisfies $\varphi_j=0$ in $B_{r_j}(x_j)$, $\varphi_j=1$ outside $B_{2r_j}(x_j)$ and $\int|\nabla_\tau\varphi_j|^2\le Cr_j^{d-3}$.

 If we take
$\varphi=\min\{\varphi_j\},$
then it is a Lipschitz function on $\partial B_1$ with values in $[0,1]$ satisfying
\begin{equation}
\label{PropertiesOfPhi}
\varphi=1 \hem\text{ outside }B_{2\delta}(\Sing(U)), \hem \mathrm{spt}(\varphi)\subset\partial B_1\backslash\Sing(U),  
\hem\text{ and }\hem
\int_{\partial B_1}|\nabla_\tau\varphi|^2\le C\eps
\end{equation}
for a dimensional constant $C$. 
With an abuse of notation, we denote the $0$-homogeneous extension of $\varphi$ by the same notation.

\vem

 Take $A$ to be the annulus region
 $$
 A=B_2\backslash B_1, 
 $$
then we have
\begin{align}
\label{AlignA}
\int_{A\cap\POSS}\mathrm{div}(\varphi^2D^2U\nabla U)=&-\int_{A\cap\Reg(U)}\varphi^2D^2U\nabla U\cdot\nu\\ 
&-\int_{\partial B_1\cap\POSS}\varphi^2D^2U\nabla U\cdot\frac{x}{|x|}+\int_{\partial B_2\cap\POSS}\varphi^2D^2U\nabla U\cdot\frac{x}{|x|},\nonumber
\end{align}
where $\nu$ denotes the inner unit normal on $\partial\POSS$. 

The homogeneity of $U$ implies $D^2U\nabla U\cdot x=\frac12\nabla|\nabla U|^2\cdot x=0$. Thus  the second line of \eqref{AlignA} vanish. For  the first line, we apply Proposition~\ref{PropEL} to identify $\nu$ with $\nabla U$. Together wtih \eqref{EqnMeanCurvature} and Lemma~\ref{LemGlobalQuant}, this gives 
$
D^2U\nabla U\cdot\nu=-H.
$
Consequently, we have
\begin{equation}
\label{AlignA'}
\int_{A\cap\POSS}\mathrm{div}(\varphi^2D^2U\nabla U)=\int_{A\cap\Reg(U)}\varphi^2 H\le\int_{A\cap\Reg(U)} H.
\end{equation}

\vem

The harmonicity of $U$ in $\POSS$ gives 
$$
\mathrm{div}(\varphi^2D^2U\nabla U)=2\varphi D^2U\nabla U\cdot\nabla\varphi+\varphi^2|D^2U|^2 \hem\text{ in }\POSS.
$$
Meanwhile, for any constant $\eta\in(0,1)$, we have
$$
|2\varphi D^2U\nabla U\cdot\nabla\varphi|\le\eta\varphi^2|D^2U|^2+\eta^{-1}|\nabla U|^2|\nabla\varphi|^2.
$$
Combining these with Lemma~\ref{LemGlobalQuant}, we have
$$
\mathrm{div}(\varphi^2D^2U\nabla U)\ge (1-\eta)\varphi^2|D^2U|^2-\eta^{-1}|\nabla\varphi|^2\hem\text{ in }\POSS.
$$

Using \eqref{PropertiesOfPhi}, we deduce
$$
\int_{A\cap\POSS}\mathrm{div}(\varphi^2D^2U\nabla U)\ge(1-\eta)\int_{A\cap\POSS}\varphi^2|D^2U|^2-C\eta^{-1}\eps
$$
for a dimensional constant $C$. 
Combined with \eqref{AlignA'}, this gives
$$
(1-\eta)\int_{A\cap\POSS}\varphi^2|D^2U|^2-C\eta^{-1}\eps\le\int_{A\cap\Reg(U)} H.
$$

Sending $\eps,\delta\to0$,  this gives
$$
(1-\eta)\int_{A\cap\POSS}|D^2U|^2\le\int_{A\cap\Reg(U)} H
$$
where we used \eqref{PropertiesOfPhi}. Since $\eta\in(0,1)$ is arbitrary, we have
$$
\int_{A\cap\POSS}|D^2U|^2\le\int_{A\cap\Reg(U)} H.
$$
The desired conclusion follows from the homogeneity of $U$. 
\end{proof}

Combining Lemma~\ref{LemLowerBoundOnHessian} and Lemma~\ref{LemMeanCurvatureControlsHessian}, we have
\begin{cor}
\label{CorLowerBoundOnH}
For $U\in\SC$, there is a dimensional constant $\delta_d>0$ such that 
$$
\int_{\Reg(U)^S}H\ge\delta_d.
$$
\end{cor}

%\begin{proof}
%It suffices to consider the case when the left-hand side is finite. This implies, via Lemma~\ref{LemMeanCurvatureControlsHessian}, that $|D^2U|\in L^2(\POSS^S)$.

%Define $\zeta=1-|\nabla U|^2$, Proposition~\ref{PropEL} gives
%$$
%\zeta=0 \hem\text{ on }\hem\partial\POSS. 
%$$
%Meanwhile, for any unit vector $\tau$ tangential to the sphere, Lemma~\ref{LemGlobalQuant} gives
%$$
%|\nabla\zeta\cdot\tau|=|2D^2U\nabla U\cdot\tau|\le2|D^2U|.
%$$
%Thus $\zeta\in H^1_0(\POSS^S)$.

%Since $U$ is a $1$-homogeneous positive harmonic function in $\POSS$ with $0$ boundary data, the principal eigenvalue of the spherical Laplacian on $\POSS^S$ is $(d-1)$. Thus we can use H\"older's inequality to get
%$$
%\int_{\POSS^S}|D^2U|^2\ge\frac14\int_{\POSS^S}|\nabla_\tau\zeta|^2\ge\frac{d-1}{4}\int_{\POSS^S}\zeta^2\ge\frac{d-1}{4}\frac{1}{|\POSS^S|}(\int_{\POSS^S}\zeta)^2.
%$$ 

%The conclusion follows from Lemma~\ref{LemStabilityGradientGap} and Lemma~\ref{LemMeanCurvatureControlsHessian}.
%\end{proof} 

As a consequence, we get a dimensional lower bound for principal eigenvalues of the Jacobi operator on singular cones. 
\begin{thm}
\label{ThmDimensionalLowerBound}
For $\Lambda_d$ in \eqref{EqnLambdaD}, we have
$$
\Lambda_d\ge\delta_d
$$
for a dimensional constant $\delta_d>0$.
\end{thm} 
\begin{rem}
A positive dimensional lower bound suffices for our purpose in this work. We remark, however, that it is an important question to quantify the sharp lower bound and to characterize the cones achieving this bound. See the discussions in Subsection~\ref{SubsectionPrincipalEigenValue}.
\end{rem} 

\begin{proof}
For $U\in\SC$, we take $\varphi$ constructed in the proof of Lemma~\ref{LemMeanCurvatureControlsHessian}, satisfying \eqref{PropertiesOfPhi} for $\eps,\delta>0$ to be chosen. Then the quotient from \eqref{EqnQuotient} satisfies
$$
\mathcal{Q}_U(\varphi)\le\frac{C\eps-\int_{\Gamma(U)\backslash B_{\delta}(\Sing(U))}H}{|\partial B_1|}
$$
for a dimensional constant $C$. 

By choosing $\eps, \delta$ small, Corollary~\ref{CorLowerBoundOnH} and \eqref{EqnVariationalPbOnSph} imply
$$
\lambda(U)\ge\delta_d
$$
for a dimensional constant $\delta_d>0$.

Since this holds for all $U\in\SC$, the desired conclusion follows. 
\end{proof} 

To use the tools from Section~\ref{SectionRegularityScales}, we localize this estimate to the collection $\RS(\rho)$ from Definition~\ref{DefRegularityScales}.
\begin{prop}
\label{PropEigenvalueInR}
For $U\in\SC$ and given $\kappa\in(0,1)$, we can find a parameter $\rho>0$, depending only on $d$ and $\kappa$, and a smooth function $\varphi:\partial B_1\to[0,1]$ satisfying
$$
\mathrm{spt}(\varphi)\cap\overline{\POSS}\subset\RS_U(\rho)
\hem\text{ and }\hem
\mathcal{Q}_U(\varphi)\le-(1-\kappa)\Lambda_d.
$$
Here $\Lambda_d>0$ is from \eqref{EqnLambdaD} and $\mathcal{Q}_U$ is the quotient from \eqref{EqnQuotient}.
\end{prop}
\begin{proof}
For a given $\kappa\in(0,1)$, suppose that there is no such $\rho>0$, then we find a sequence $U_n\in\SC$ such that whenever $\varphi:\partial B_1\to[0,1]$ is a smooth function with 
$$
\mathcal{Q}_{U_n}(\varphi)\le-(1-\kappa)\Lambda_d,
$$
we must have
\begin{equation}
\label{MustHave}
\mathrm{spt}(\varphi)\cap\overline{\{U_n>0\}}\backslash\RS_{U_n}(1/n)\neq\emptyset.
\end{equation}

\vem

Up to a subsequence, Lemma~\ref{LemCompactness} gives $U_n\to U\in\SC$ locally uniformly in $L^\infty$ and in $H^1$. For this $U$, the proof of Theorem~\ref{ThmDimensionalLowerBound} gives a smooth function $\varphi:\partial B_1\to[0,1]$ with 
\begin{equation}
\label{EqnQUQU}
\mathcal{Q}_U(\varphi)\le -(1-\kappa/2)\Lambda_d.
\end{equation}
and
$$
\mathrm{spt}(\varphi)\subset\partial B_1\backslash\Sing(U).
$$
With Proposition~\ref{PropOpenCover} and the compactness of $\mathrm{spt}(\varphi)$, we find $\rho>0$ such that 
\begin{equation}
\label{TheSupportIsGood}
\spt(\varphi)\cap\overline{\POSS}\subset\RS_U(\rho).
\end{equation} 

\vem

With Lemma~\ref{LemIOF}, Lemma~\ref{LemCompactness},  Lemma~\ref{LemFollowingOpenCover} and Lemma~\ref{LemLimInf}, we see that $\{U_n>0\}$ converges to $\POSS$ in $C^{2,\alpha}$ on $\spt(\varphi)$. It follows from \eqref{EqnQUQU} that 
$$
\mathcal{Q}_{U_n}(\varphi)\le-(1-\kappa)\Lambda_d
$$
for large $n$. With \eqref{MustHave}, we find 
$$
p_n\in\spt(\varphi)\cap\overline{\{U_n>0\}}\backslash\RS_{U_n}(1/n).
$$

Up to a subsequence, we have
$$
p_n\to p\in\spt(\varphi)\cap\overline{\POSS}.
$$
Theorem~\ref{ThmContinuityOfRS} gives $\rho_U(p)=0,$ contradicting \eqref{TheSupportIsGood}.
\end{proof} 

Finally, we relate these estimates to an eigenvalue problem in the collection of points with bounded regularity scales:
\begin{cor}
\label{CorAlmostEigenFunction}
For $U\in\SC$ and given $\kappa\in(0,1)$, there is $\rho>0$, depending only on $d$ and $\kappa$, such that we can find a nonnegative function $\varphi:\partial B_1\to\R$ satisfying
$$
\spt(\varphi)\cap\overline{\POSS}\subset\RS(\rho)
$$
and 
$$\begin{cases}
\Delta_{\Sph}\varphi\ge(1-\kappa)\Lambda_d \varphi &\text{ in }\POSS^S,\\
\varphi_\nu+H\varphi\ge0 &\text{ on }\Gamma(U)^S,
\end{cases}$$
where $\Lambda_d>0$ is from \eqref{EqnLambdaD}.
\end{cor}
Recall the space of singular cones $\SC$ from \eqref{EqnCones} and the notation for spherical intersection from \eqref{EqnES}. Here $\nu$ denotes the inner unit normal to the free boundary $\Gamma(U)$, and $H$ denotes the mean curvature as in \eqref{EqnMeanCurvature}.

\begin{proof}
For given $\kappa$, Proposition~\ref{PropEigenvalueInR} gives $\rho>0$, depending only on $d$ and $\kappa$, and a smooth function $\tilde{\varphi}:\partial B_1\to[0,1]$ with $\spt(\tilde{\varphi})\cap\overline{\POSS}\subset\RS(\rho)$ such that 
\begin{equation}
\label{1}
\mathcal{Q}_U(\tilde{\varphi})\le-(1-\kappa)\Lambda_d.
\end{equation}
The compactness of $\spt(\tilde{\varphi})$ and the openness of $\RS(\rho)$ (see Lemma~\ref{LemBasicPropertiesOfRS}) give a smooth domain $\Omega\subset\partial B_1$ such that $\partial\Omega$ intersects $\Gamma(U)$ transversally and 
$$
\spt(\tilde{\varphi})\subset\Omega,\hem\text{ and }\hem\Omega\cap\overline{\POSS}\subset\RS(\rho).
$$

\vem

The direct method gives a minimizer $\varphi$ to the following minimization problem
$$
-\lambda:=\inf\{\mathcal{Q}_U(f):\hem f\in H^1_0(\Omega)\}.
$$
Replacing $\varphi$ with its absolute value if necessary, we can assume $\varphi\ge0$. The minimizer solves the Euler-Lagrange equation
$$\begin{cases}
\Delta_{\Sph}\varphi=\lambda \varphi &\text{ in }\Omega\cap\POSS,\\
\varphi_\nu+H\varphi=0 &\text{ on }\Omega\cap\Gamma(U).
\end{cases}$$

Extending $\varphi$ to the entire $\partial B_1$ by $0$, we get a subsolution to this system outside $\Omega$. With \eqref{1}, we see that $\lambda\ge(1-\kappa)\Lambda_d$. 
\end{proof}

%%%%%%%%%%%%%%%%%%%%%%%%%%%%%%%%%%%%%%%%%%%%%%%%%%%%%%%%%%%%%%%%%%%%%%%%%%%%%%%%%%%%%%%%%%%%%%%%%%%%%%%%%%%%%%%%%%%%%%%%%%%%%%%%%%%%%%%%%%%%%%%%%%%%
\subsection{Decay of positive Jacobi fields}
For  $\Lambda_d$ from \eqref{EqnLambdaD}, the bound in \eqref{PropStabilityEigenBound} implies the existence of real roots for the following equation 
$$
\gamma(\gamma-d+2)+\Lambda_d=0.
$$
Define $\gamma_d$ as the smaller root, namely,
\begin{equation}
\label{EqnGammaD}
\gamma_d:=\frac{d-2}{2}-\sqrt{\left(\frac{d-2}{2}\right)^2-\Lambda_d}\in\left(0,\frac{d-2}{2}\right].
\end{equation} 
The range of $\gamma_d$ follows from Theorem~\ref{ThmDimensionalLowerBound}.

Based on Subsection~\ref{SubsectionJacobi}, the value of $\gamma_d$ dictates the rate of decay for positive Jacobi fields. We start with a weak estimate, inspired by \cite{Wa}.
\begin{lem}
\label{LemWeakDecayJacobiField}
Suppose that $\omega$ is a positive Jacobi field on $U\in\SC$, and that $\gamma_d$ is from \eqref{EqnGammaD}.

For given $\gamma\in(0,\gamma_d)$, there is $\bar{\rho}>0$, depending only on $d$ and $\gamma$, such that 
$$
\inf_{\partial B_r\cap\RS_U(r\rho)}\omega\le r^{-\gamma}\sup_{\partial B_1\cap\RS_U(\rho)}\omega
$$
for all $r\ge1$ and $0<\rho<\bar{\rho}.$
\end{lem} 
Recall notations from \eqref{EqnCones}, Definitions~\ref{DefPostiveJacobiField} and \ref{DefRegularityScales}.
\begin{proof}
With the ordering between $\RS(\rho)$ from Lemma~\ref{LemBasicPropertiesOfRS}, it suffices to prove the estimates for one $\rho>0$.

For given $\gamma\in(0,\gamma_d)$, we take
\begin{equation}
\label{LambdaTemp}
\lambda:=-\gamma(\gamma-d+2).
\end{equation} 
Then $0<\lambda<\Lambda_d$. As a result, Corollary~\ref{CorAlmostEigenFunction} gives $\rho>0$, depending only on $d$ and $\gamma$, and a nonnegative function $\varphi$ on $\partial B_1$ with 
\begin{equation}
\label{LalalaSPT}
\spt(\varphi)\cap\overline{\POSS}\subset\RS(\rho)
\end{equation}
and 
\begin{equation}
\label{TempSubsolutionSystem}
\begin{cases}
\Delta_{\Sph}\varphi\ge\lambda\varphi &\text{ in }\POSS^S,\\
\varphi_\nu+H\varphi\ge0 &\text{ on }\Gamma(U)^S.
\end{cases}
\end{equation}
Recall our notation for spherical intersections from \eqref{EqnES}.

\vem

For $r>0$, define the following  quantity
$$
I(r):=\int_{\POSS^S}\omega(r\theta)\varphi(\theta)dH^{d-1}(\theta).
$$
Note that the integrand is supported on $\Reg(U)$, which allows differentiation of $I(\cdot)$:
$$
I'(r)=\int_{\POSS^S}\frac{\partial}{\partial r}\omega (r\theta)\varphi(\theta) dH^{d-1}(\theta),\hem\text{ and }I''(r)=\int_{\POSS^S}\frac{\partial^2}{\partial r^2}\omega(r\theta)\varphi(\theta)dH^{d-1}(\theta).
$$
As a result, we have
\begin{align*}
I''(r)+\frac{d-1}{r}I'(r)&=\int_{\POSS^S}\left[\frac{\partial^2 }{\partial r^2}+\frac{d-1}{r}\frac{\partial }{\partial r}\right]\omega(r\theta)\varphi(\theta) dH^{d-1}(\theta)\\
&=-\frac{1}{r^2}\int_{\POSS^S}\Delta_{\Sph}\omega(r\theta)\,\varphi(\theta) dH^{d-1}(\theta)
\end{align*}
since $\Delta\omega=0$ in $\POSS$. 
An integration by parts and \eqref{TempSubsolutionSystem} imply, denoting $\bar \omega (\theta) = \omega(r \theta)$,
\begin{align*}
\int_{\POSS^S}\Delta_{\Sph}\bar \omega\,\varphi&=\int_{\POSS^S}\bar \omega\Delta_{\Sph}\varphi-\int_{\Gamma(U)^S}\bar \omega_\nu\varphi+\int_{\Gamma(U)^S}\bar \omega\varphi_\nu \ge\lambda\int_{\POSS^S}\bar\omega\varphi.
\end{align*}
Consequently, we have
\begin{equation}
\label{EqnDifferentialInequalityI}
I''(r)+\frac{d-1}{r}I'(r)\le -\frac{\lambda}{r^2}I(r) \hem\text{ for all }r>0.
\end{equation}

\vem

To use this differential inequality, for $\alpha>0$, we define for $t>0$
$$
J(t):=t^{-\gamma\alpha}I(t^{-\alpha}).
$$
Direct differentiation gives
$$
J'(t)=-\gamma\alpha t^{-\gamma\alpha-1}I(t^{-\alpha})-\alpha t^{-\gamma\alpha-\alpha-1}I'(t^{-\alpha})
$$
and
$$
J''(t)=\alpha^2t^{-\gamma\alpha-2\alpha-2}\left\{I''(t^{-\alpha})+\frac{2\gamma+\alpha^{-1}+1}{t^{-\alpha}}I'(t^{-\alpha})+\frac{\gamma(\gamma+\alpha^{-1})}{t^{-2\alpha}}I(t^{-\alpha})\right\}.
$$

We pick $\alpha$ such that 
$$
\alpha^{-1}=d-2-2\gamma.
$$
Note that $\gamma<\gamma_d$ in \eqref{EqnGammaD} implies $\alpha>0$. We have, with $r=t^{-\alpha}$,
$$
J''(t)=\alpha^2t^{-\gamma\alpha-2\alpha-2}\left\{I''+\frac{d-1}{r}I'+\frac{\lambda}{r^2}I\right\}\le0 \hem\text{ for all }t>0,
$$
where we used \eqref{LambdaTemp} and \eqref{EqnDifferentialInequalityI}.

Since $J$ stays nonnegative on $(0,+\infty)$, this implies that $J'\ge0$. In terms of $I$, this gives
$
rI'(r)+\gamma I(r)\le 0.
$
Or, equivalently,
$$
(r^\gamma I)'\le0 \hem\text{ for all }r>0.
$$

\vem
For $r\ge1,$ we have $r^\gamma I(r)\le I(1)$, that is,
$$
\int_{\POSS^S}\omega(r\theta)\varphi(\theta)dH^{d-1}(\theta)\le r^{-\gamma}\int_{\POSS^S}\omega(\theta)\varphi(\theta)dH^{d-1}(\theta).
$$
With \eqref{LalalaSPT} and Lemma~\ref{LemScalingRS}, we have
$$
\int_{\POSS^S}\omega(r\theta)\varphi(\theta)dH^{d-1}(\theta)\ge \inf_{\partial B_r\cap\RS(r\rho)}\omega \int_{\POSS^S}\varphi
$$
and
$$
\int_{\POSS^S}\omega(\theta)\varphi(\theta)dH^{d-1}(\theta)\le\sup_{\partial B_1\cap\RS(\rho)}\omega \int_{\POSS^S}\varphi.
$$
The conclusion follows.
\end{proof} 

Lemma~\ref{LemWeakDecayJacobiField} can be upgraded into a strong estimate with the Harnack inequality in Theorem~\ref{ThmHarnack}:
\begin{prop}
\label{PropDecayJacobiField}
Suppose that $\omega$ is a positive Jacobi field on $U\in\SC$, and that $\gamma_d$ is from \eqref{EqnGammaD}.

For given $\gamma\in(0,\gamma_d)$, there are constants $\rho$ and $C$, depending only on $d$ and $\gamma$, such that 
$$
\sup_{\partial B_r\cap\RS_U(r\rho)}\omega\le C r^{-\gamma}\inf_{\partial B_1\cap\RS_U(\rho)}\omega
$$
for all $r\ge1.$
\end{prop}

%%%%%%%%%%%%%%%%%%%%%%%%%%%%%%%%%%%%%%%%%%%%%%%%%%%%%%%%%%%%%%%%%%%%%%%%%%%%%%%%%%%%%%%%%%%%%%%%%%%%%%%%%%%%%%%%%%%%%%%%%%%%%%%%%%%%%%%%%%%%%%%%%%%%%%%%%%%%%%%%%%%%%%%%%%%%%%%%%%%%%%%%%%%%%%%%%%%%%%%%%%%%%%%%%%%%%%%%%%%%%%
\section{Separation between ordered minimizers}
\label{sec:sep}
At small scales, the separation between ordered minimizers in $\OM(\cdot)$ (see Definition~\ref{DefOrderedMinimizers}) is modeled by  positive Jacobi fields. This intuition was pointed out in the pioneering works by De Silva-Jerison-Shahgholian \cite{DJS} and Edelen-Spolaor-Velichkov \cite{EdSV}.

We begin with a linearization lemma in a form that is convenient to us:
\begin{lem}
\label{LemLinearization}
Suppose that $U\in\M(\R^d)$ satisfies
$
B_{r_d}(e_1)\subset\POSS\cap\RS_U(\rho_d)
$
for $r_d,\rho_d$ from Lemma~\ref{LemAnchoringPoint}.

For a sequence $(u_n,v_n)\in\OM(B_n)$ satisfying
$$
u_n,v_n\to U \hem\text{ locally uniformly in }\R^d,
$$
define
$$
\varphi_n:=\frac{v_n-u_n}{(v_n-u_n)(e_1)}.
$$

Then for positive $R$ and $\rho$, there is a constant $C>1$, depending on $d$, $R$, $\rho$ and $U$, such that 
\begin{equation}
\label{EqnLinearizationLip}
\|\varphi_n\|_{C^1(B_R\cap\RS_{u_n}(\rho))}\le C \hem\text{ and }\hem C^{-1}\le\inf_{B_R\cap\RS_{u_n}(\rho)}\varphi_n\le\sup_{B_R\cap\RS_{u_n}(\rho)}\varphi_n\le C
\end{equation} 
for all $n\in\N$.

Moreover, up to a subsequence, we have
$$
\varphi_n\to\varphi \hem\text{ locally uniformly in }C^2(\POSS),
$$
where $\varphi$ is a positive Jacobi field on $U$ as in Definition~\ref{DefPostiveJacobiField}.
\end{lem} 
Recall the space of minimizers $\M(\cdot)$ and the space of ordered minimizers $\OM(\cdot)$ from \eqref{EqnM} and Definition~\ref{DefOrderedMinimizers}. 
Collection of points with controlled regularity scales, $\RS(\cdot)$, is defined in Definition~\ref{DefRegularityScales}. 

\begin{proof}
The compactness of $\{\varphi_n\}$ and properties of the  limit $\varphi$ were established in Theorem 6.2 of \cite{EdSV}. Below we prove \eqref{EqnLinearizationLip} for $R=1$ and a given $\rho>0$.

Suppose the estimates in \eqref{EqnLinearizationLip} fail, then up to a subsequence, we can find $$p_n\in B_1\cap\RS_{u_n}(\rho)$$ such that 
\begin{equation}
\label{ANOTHERCONTRA}
|\nabla\varphi_n(p_n)|+|\log\varphi_n(p_n)|\to\infty. 
\end{equation}
Picking a further subsequence, we apply Theorem~\ref{ThmContinuityOfRS} to get 
$$
p_n\to p\in\overline{ B_1}\cap\RS_U(\rho/2).
$$

Depending on whether $p\in\POSS$ or $p\in\Gamma(U)$, we seek a contradiction in two cases. The argument is similar to the proofs of Lemma~\ref{LemConnectivity} and Theorem~\ref{ThmHarnack}, thus we only sketch the main ideas.

\vem

\textit{Case 1: $p\in\POSS.$}

In this case, we find $\delta>0$ such that $B_{5\delta}(p)\subset\POSS$. With the connectedness of $\POSS$ from Theorem~\ref{ThmGlobalConnected}, there is a continuous curve in the interior of $\POSS$ that connects $p$ to $e_1$. With $u_n,v_n\to U$, this curve stays in the interior of $\{u_n>0\}\subset\{v_n>0\}$ for large $n$. In particular, the difference $(v_n-u_n)$ is a positive harmonic function in a neighborhood of this curve. 

Consequently, Harnack inequality for harmonic functions gives a constant $C>0$ such that $C^{-1}\le \frac{v_n-u_n}{(v_n-u_n)(e_1)}\le C$ in a neighborhood of this curve. In terms of $\varphi_n$, this gives
$$
C^{-1}\le \varphi_n\le C \hem\text{ in }B_{4\delta}(p).
$$
For large $n$, we have $p_n\in B_{\delta}(p)$. Gradient estimates for harmonic functions gives
$$
|\nabla\varphi_n(p_n)|+|\log\varphi_n(p_n)|\le C \hem\text{ for large }n,
$$
contradicting \eqref{ANOTHERCONTRA}.

\vem

\textit{Case 2: $p\in\Gamma(U).$}

In this case, Lemma~\ref{LemIOF} and Lemma~\ref{LemFollowingOpenCover} imply that, up to a rotation,  $\POSS\cap B_{\mu_d\rho}(p)$ and $\{u_n>0\}\cap B_{\mu_d\rho}(p)$ are epigraphs of $C^{2,\alpha}$ functions in the $x_d$-direction. Moreover, if we take $\bar{p}=p+\frac12\mu_d\rho e_d$, then $U(\bar{p})\ge c\rho$. 
A similar argument as in Case 1 gives a constant $C>0$ such that 
$$
C^{-1}\le \varphi_n(\bar{p})\le C
$$
for large $n$. 

With $U(\bar{p})\ge c\rho$, we have $\bar{p}\in\{u_n>0\}\cap B_{\frac12\mu_d\rho}(p)$ for all large $n$. Thus we can apply Lemma~\ref{LemSchauder} to conclude
$$
|\nabla\varphi_n(p_n)|+|\log\varphi_n(p_n)|\le C \hem\text{ for large }n,
$$
contradicting \eqref{ANOTHERCONTRA}.
\end{proof} 

With this, we give an unconditional control on the growth of the separation between ordered minimizers:
\begin{lem}
\label{LemGrowthControl}
Given $A>2$ and $\rho\in(0,\rho_d)$, where $\rho_d$ is from Lemma~\ref{LemAnchoringPoint}, there is a constant $L>2A$, depending only on $d$, $A$ and $\rho$, such that the following holds:

For $(u,v)\in\OM(B_L)$ with 
$$0\in\GU\hem\text{ and }\hem\Gamma(v)\cap B_1\neq\emptyset,$$ we have
$$
\sup_{\partial B_A\cap\RS_u(A\rho)}(v-u)\le L\sup_{\partial B_1\cap\RS_u(\rho)}(v-u).
$$
\end{lem}

\begin{proof}
For given $A$ and $\rho\in(0,\rho_d)$, suppose that there is no such $L$, we find a sequence $(u_n,v_n)\in\OM(B_n)$ with $0\in\Gamma(u_n)$, $\Gamma(v_n)\cap B_1\neq\emptyset$ and 
\begin{equation}
\label{AnotherContra}
\sup_{\partial B_A\cap\RS_{u_n}(A\rho)}(v_n-u_n)>n \sup_{\partial B_1\cap\RS_{u_n}(\rho)}(v_n-u_n).
\end{equation}
Proposition~\ref{PropLipNonD} implies the uniform boundedness of the left-hand side, thus
\begin{equation}
\label{AnotherContra'}
\sup_{\partial B_1\cap\RS_{u_n}(\rho)}(v_n-u_n)\to0.
\end{equation}

Up to a subsequence, Lemma~\ref{LemCompactness} gives
$$
u_n\to U\in\M(\R^d) \hem\text{ and }\hem v_n\to V\in\M(\R^d)\hem\text{ locally uniformly. }
$$
Up to a rotation, Lemma~\ref{LemAnchoringPoint} gives
$$
B_{r_d}(e_1)\subset\POSS\cap\RS_U(\rho_d).
$$
With Theorem~\ref{ThmContinuityOfRS}, this gives
$$
e_1\in\RS_{u_n}(\rho)
$$
for all large $n$.

It follows from \eqref{AnotherContra'} that $U(e_1)=V(e_1)$.  Theorem~\ref{ThmStricMax} implies that 
$$
U=V \hem\text{ in }\R^d.
$$
As a result, we can apply Lemma~\ref{LemLinearization} with $R=A$ in \eqref{EqnLinearizationLip} to get
$$
\sup_{\partial B_A\cap\RS_{u_n}(A\rho)}(v_n-u_n)\le \sup_{ B_A\cap\RS_{u_n}(\rho)}(v_n-u_n)\le C(v_n-u_n)(e_1)
$$
With $e_1\in\RS_{u_n}(\rho)$ for large $n$, we have
$$
\sup_{\partial B_1\cap\RS_{u_n}(\rho)}(v_n-u_n)\ge(v_n-u_n)(e_1).
$$

Combining these two estimates, we get a contradiction to \eqref{AnotherContra}.
\end{proof} 

Using the singularity structure, the separation between ordered minimizers decays:
\begin{lem}
\label{LemDecay}
Given $\gamma\in(0,\gamma_d)$ for $\gamma_d$ from \eqref{EqnGammaD}, there are constants $\rho$ and $\delta$ small,  and $A$ large, depending only on $d$ and $\gamma$, such that the following holds:

For $(u,v)\in\OM(B_{\delta^{-1}})$ with 
$$
0\in\Sing(u),\hem\Gamma(v)\cap B_\delta\neq\emptyset
$$
and
$$
W(u;0,2)-W(u;0,1)<\delta,
$$
we have
$$
\sup_{\partial B_A\cap\RS_u(A\rho)}(v-u)\le A^{-\gamma}\sup_{\partial B_1\cap\RS_u(\rho)}(v-u).
$$
\end{lem} 
Recall  the Weiss monotonicity formula $W(\cdot)$ from \eqref{EqnWeiss}.

\begin{proof}
For given $\gamma\in(0,\gamma_d)$, let $\gamma'=\frac12(\gamma+\gamma_d)<\gamma_d$. Fix $\rho>0$ as in Proposition~\ref{PropDecayJacobiField}, corresponding to $\gamma'$. Let $C_{\gamma'}$ denote the constant $C$ from the same proposition. Finally, we pick $A>2$ such that 
\begin{equation}
\label{aNoTher}
A^{\frac12(\gamma_d-\gamma)}>2^{1+\gamma'}C_{\gamma'}.
\end{equation}

With such  $\rho$ and $A$, suppose that the statement fails for any $\delta>0$, we find a sequence $(u_n,v_n)\in\OM(B_n)$ with 
\begin{equation}
\label{WeissPinching}
0\in\Sing(u_n), \hem\Gamma(v_n)\cap B_{1/n}\neq\emptyset,
\hem\text{ and }\hem
W(u_n;0,2)-W(u_n;0,1)<1/n,
\end{equation} 
but
\begin{equation}
\label{anothercontra}
\sup_{\partial B_A\cap\RS_{u_n}(A\rho)}(v_n-u_n)> A^{-\gamma}\sup_{\partial B_1\cap\RS_{u_n}(\rho)}(v_n-u_n).
\end{equation} 

Up to a subsequence, Lemma~\ref{LemCompactness} implies that 
$$
u_n\to U\in\M(\R^d)\hem\text{ and }\hem v_n\to V\in\M(\R^d).
$$
With \eqref{WeissPinching}, we see that $0\in\Gamma(U)\cap\Gamma(V)$. It follows from Theorem~\ref{ThmStricMax} that 
$$
U=V \hem\text{ in }\R^d.
$$
With \eqref{WeissPinching}, we have $W(U;0,2)=W(U;0,1)$. It follows from Proposition~\ref{PropWeiss} that 
$$
U\in\SC.
$$
Up to a rotation, Lemma~\ref{LemAnchoringPoint} gives
$$
B_{r_d}(e_1)\subset\POSS\cap\RS_U(\rho_d).
$$
In particular, Theorem~\ref{ThmContinuityOfRS} implies 
\begin{equation}
\label{E1admissible}
e_1\in\RS_{u_n}(\rho)\hem \text{ for all large }n.
\end{equation}

For each $n$, define 
$$
\varphi_n:=\frac{v_n-u_n}{(v_n-u_n)(e_1)},
$$
and let $\varphi$ denote the subsequential limit of $\{\varphi_n\}$ from Lemma~\ref{LemLinearization}.

With \eqref{anothercontra} and \eqref{E1admissible}, we have
$$
\varphi_n(p_n)>A^{-\gamma} \hem\text{ at some }\hem p_n\in\partial B_A\cap\RS_{u_n}(A\rho).
$$
For $\eps>0$ small, Lemma~\ref{LemBasicPropertiesOfRS} gives $B_\eps(p_n)\subset\RS_{u_n}(\frac34A\rho)$. Proposition~\ref{PropLipNonD} gives $\bar{p}_n\in B_\eps(p_n)$ such that 
$$
u_n(\bar{p}_n)\ge c\eps, 
\hem\text{ and }\hem
\varphi_n(\bar{p}_n)>A^{-\gamma}/2,
$$
where the second comparison follows from \eqref{EqnLinearizationLip}.

As a result, we have $B_{c\eps}(\bar{p}_n)\subset \{u_n>0\}\cap\POSS$ for large $n$.
Up to a subsequence, we have
$$
\bar{p}_n\to \bar{p}\in B_{A-\eps}^c\cap\RS_U(A\rho/2),
$$
where we used Theorem~\ref{ThmContinuityOfRS}.

With the convergence of $\varphi_n$ to $\varphi$ in $\POSS$, we conclude
$$
\varphi(\bar{p})\ge A^{-\gamma}/2 \hem\text{ at some }\bar{p}\in B_{A-\eps}^c\cap\RS_U(A\rho/2)\subset B_{A/2}^c\cap\RS_U(A\rho/2).
$$
On the other hand, with $\varphi_n(e_1)=1$ for all $n$, we have
$
\varphi(e_1)=1
$
with $e_1\in\RS_U(\rho)$.
\vem

Consequently, Proposition~\ref{PropDecayJacobiField} gives
\begin{align*}
A^{-\gamma}/2\le\varphi(\bar{p})\le\sup_{B^c_{A/2}\cap\RS_U(A\rho/2)}\varphi\le C_{\gamma'}(A/2)^{-\gamma'}\inf_{\partial B_1\cap\RS_U(\rho)}\varphi\le C_{\gamma'}(A/2)^{-\gamma'}.
\end{align*}
This contradicts \eqref{aNoTher}.
\end{proof}

We combine Lemma~\ref{LemGrowthControl} and Lemma~\ref{LemDecay} to get the main technical estimate of this work.
\begin{prop}
\label{PropMain}
For $\gamma_d$ from \eqref{EqnGammaD}, let $\gamma\in(0,\gamma_d)$. 

There are positive constants $\rho$ and $\delta$ small, and $C$ large, depending only on $d$ and $\gamma$, such that for $(u,v)\in\OM(B_1)$ with 
$$
0\in\Sing(u) \hem\text{ and }\hem \Gamma(v)\cap B_s\neq\emptyset\hem \text{ for some }0<s<1/2,
$$
we have
$$
\sup_{\partial B_{\delta}\cap\RS_u(\delta\rho)}(v-u)\le Cs^{1+\gamma}.
$$
\end{prop} 
Recall the space of ordered minimizers $\OM(\cdot)$ from Definition~\ref{DefOrderedMinimizers}. The collection of points with controlled regularity scales, $\RS(\cdot)$, is given in Definition~\ref{DefRegularityScales}.

\begin{proof}
For $\gamma\in(0,\gamma_d)$, let $\rho,\delta$ and $A$ be  constants from Lemma~\ref{LemDecay}. For such $\rho$ and $A$, let $L$ denote the constant from Lemma~\ref{LemGrowthControl}. Without loss of generality, we can assume $L<\delta^{-1}.$

With Proposition~\ref{PropLipNonD}, choosing $C$ large, it suffices to consider small $s$. The desired estimate follows from an iteration process which we lay out in several steps.

\vem

\textit{Step 1: The initial normalization.}

Choose $r:=s/\delta$, and let 
$$
u_0(x):=\frac{u(rx)}{r} \hem\text{ and }v_0(x):=\frac{v(rx)}{r}. 
$$
The $(u_0,v_0)\in\OM(B_{\delta/s})$ with $0\in\Sing(u_0)$ and $\Gamma(v_0)\cap B_\delta\neq\emptyset.$

If we have $\delta/s>\delta^{-1}$, then depending on the Weiss energy for $u_0$ from \eqref{EqnWeiss}, we have the  dichotomy:
\begin{enumerate}
\item{Either $W(u_0;0,2)-W(u_0;0,1)\ge\delta$, then we have, by Lemma~\ref{LemGrowthControl}, 
$$
\sup_{\partial B_A\cap\RS_{u_0}(A\rho)}(v_0-u_0)\le L\sup_{\partial B_1\cap\RS_{u_0}}(v_0-u_0);
$$}
\item{Or $W(u_0;0,2)-W(u_0;0,1)<\delta$, then we have, by Lemma~\ref{LemDecay},
$$
\sup_{\partial B_A\cap\RS_{u_0}(A\rho)}(v_0-u_0)\le A^{-\gamma}\sup_{\partial B_1\cap\RS_{u_0}}(v_0-u_0).
$$}
\end{enumerate}

In either case, we define
$$
u_1(x):=\frac{u_0(Ax)}{A}\hem\text{ and }v_1(x):=\frac{v_0(Ax)}{A},
$$
then we have $(u_1,v_1)\in\OM(B_{\frac{\delta}{As}})$, $0\in\Sing(u_1)$ and $\Gamma(v_1)\cap B_\delta\neq\emptyset$.
\vem

\textit{Step 2: The inductive step.}

Suppose that we have found $\{(u_j,v_j)\}_{j=1,\dots,n}$ with 
$$
(u_n,v_n)\in\OM(B_{\frac{\delta}{A^ns}}), \hem 0\in\Sing(u_n) \hem\text{ and }\hem \Gamma(v_n)\cap B_\delta\neq\emptyset.
$$

If $\frac{\delta}{A^ns}\le\delta^{-1}$, then we terminate the iteration.  

If $\frac{\delta}{A^ns}>\delta^{-1}$, we have a dichotomy as in Step 1. Using either Lemma~\ref{LemGrowthControl} or Lemma~\ref{LemDecay}, we get
\begin{enumerate}
\item{Either $W(u_n;0,2)-W(u_n;0,1)\ge\delta$, then we have
$$
\sup_{\partial B_A\cap\RS_{u_n}(A\rho)}(v_n-u_n)\le L\sup_{\partial B_1\cap\RS_{u_n}}(v_n-u_n);
$$}
\item{Or $W(u_n;0,2)-W(u_n;0,1)<\delta$, then we have
$$
\sup_{\partial B_A\cap\RS_{u_n}(A\rho)}(v_n-u_n)\le A^{-\gamma}\sup_{\partial B_1\cap\RS_{u_n}}(v_n-u_n).
$$}
\end{enumerate}

We define 
$$
u_{n+1}(x):=\frac{1}{A}u_{n}(Ax)\hem\text{ and }v_{n+1}(x):=\frac{1}{A}v_n(Ax),
$$
satisfying $(u_{n+1},v_{n+1})\in\OM(B_{\frac{\delta}{A^{n+1}s}})$, $0\in\Sing(u_{n+1})$ and $\Gamma(v_{n+1})\cap B_\delta\neq\emptyset$.

Moreover, if we encounter possibility (1) in the dichotomy, then 
$$
\sup_{\partial B_1\cap\RS_{u_{n+1}}(\rho)}(v_{n+1}-u_{n+1})=\frac{1}{A}\sup_{\partial B_A\cap\RS_{u_n}(A\rho)}(v_n-u_n)\le \frac{L}{A}\sup_{\partial B_1\cap\RS_{u_n}(\rho)}(v_n-u_n).
$$
Otherwise, we have
$$
\sup_{\partial B_1\cap\RS_{u_{n+1}}(\rho)}(v_{n+1}-u_{n+1})\le A^{-1-\gamma}\sup_{\partial B_1\cap\RS_{u_n}(\rho)}(v_n-u_n).
$$
Note that we used the scaling symmetry from Lemma~\ref{LemScalingRS}.
\vem

\textit{Step 3: Conclusion.}

Suppose after $N$ times, we terminate the iteration. That is, we have
\begin{equation}
\label{3}
A^{N-1}<\delta^2/s \hem\text{ and }\hem A^N\ge\delta^2/s.
\end{equation}

Among these $N$ iterations, suppose that for $m$ times we encounter possibility (1) in the dichotomy at steps $j_1<j_2<j_3<\dots<j_m$.  Then we have
\begin{equation}
\label{2}
\sup_{\partial B_1\cap\RS_{u_N}(\rho)}(v_N-u_N)\le \left(\frac{L}{A}\right)^m(A^{-\gamma-1})^{N-M}\sup_{\partial B_1\cap\RS_{u_0}(\rho)}(v_0-u_0).
\end{equation}

On the other hand, for $k=1,\dots,m-1$, we have
\begin{align*}
W(u_{j_k};0,1)-W(u_{j_{k+1}};0,1)&=W(u_{j_{k+1}};0,A^{j_{k+1}-j_{k}})-W(u_{j_{k+1}};0,1)\\
&\ge W(u_{j_{k+1}};0,2)-W(u_{j_{k+1}};0,1)\\
&\ge\delta
\end{align*}
since $A>2.$
Proposition~\ref{PropWeiss} implies
$$
W(u_0;0,2)-W(u_0;0,0+)\ge \sum_{j\ge 0} [W(u_j;0,1)-W(u_{j+1};0,1)]\ge m\delta.
$$
It follows from Lemma~\ref{LemWeissGap} that $m$ is bounded by a constant, depending only on $d$ and $\delta$. Together with \eqref{2}, this gives
$$
\sup_{\partial B_1\cap\RS_{u_N}(\rho)}(v_N-u_N)\le C(A^{-1-\gamma})^N
$$
for a constant $C$, depending on $d$ and $\gamma$. 

Rescaling back to $(u,v)$, this gives
$$
\sup_{\partial B_{A^Nr}\cap\RS_u(A^Nr\rho)}(v-u)\le CA^{-\gamma N}s/\delta.
$$
Recall that $r=s/\delta$, the conclusion follows from \eqref{3}.
\end{proof}

%%%%%%%%%%%%%%%%%%%%%%%%%%%%%%%%%%%%%%%%%%%%%%%%%%%%%%%%%%%%%%%%%%%%%%%%%%%%%%%%%%%%%%%%%%%%%%%%%%%%%%%%%%%%%%%%%%%%%%%%%%%%%%%%%%%%%%%%%%%%%%%%%%%%%%%%%%%%%%%%%%%%%%%%%%%%%%%%%%%%%%%%%%%%%%%%%%%%%%%%%%%%%%%%%%%%%%%%%%%%%%
\section{Improved generic regularity}
\label{sec:concl}
In this section, we  establish our main results Theorem~\ref{ThmMainIntro} and Corollary~\ref{CorMainIntro}. Compared with our previous result, the key improvement is the new superlinear  `cleaning' estimate \eqref{EqnNewCleaningIntro}.

\begin{lem}
\label{LemCleaningNew}
Suppose that $(u,v)\in\OM(B_1)$ with 
$$
p\in\Sing(u)\cap B_{1/2}, \hem \Gamma(v)\cap B_s(p)\neq\emptyset \hem\text{ for some }0<s<1/4,
$$
and 
$$
v\ge u+\tau \hem\text{ on }\partial B_1\cap\PosS \hem\text{ for some }0<\tau<1.
$$

Given $\gamma\in(0,\gamma_d)$ with $\gamma_d$ from \eqref{EqnGammaD}, there is a constant $C$, depending only on $\gamma$ and the dimension $d$, such that 
$$
\tau\le C s^{1+\gamma} \sup_{\partial B_1}u.
$$
\end{lem} 
Recall the space of ordered minimizers $\OM(\cdot)$ from Definition~\ref{DefOrderedMinimizers}. The singular set $\Sing(\cdot)$ is given in Definition~\ref{DefRegSing}, and the free boundary $\Gamma(\cdot)$ from \eqref{EqnGU}.

\begin{proof}
Proposition~\ref{PropMain} gives
$$
\sup_{\partial B_\delta(p)\cap\RS_{u}(\delta\rho)}(v-u)\le Cs^{1+\gamma}.
$$
Meanwhile, note that $(v-u)$ and $u$ are harmonic in $\PosS$. The comparison principle gives
$$
v-u\ge\frac{\tau}{\sup_{\partial B_1}u}  u \hem\text{ in }B_1.
$$
Therefore, we have
$$
\tau \sup_{\partial B_\delta(p)\cap\RS_{u}(\delta\rho)}u\le C s^{1+\gamma}\sup_{\partial B_1}u .
$$

With Proposition~\ref{PropLipNonD} and Lemma~\ref{LemAnchoringPoint}, we bound $\sup_{\partial B_\delta(p)\cap\RS_{u}(\delta\rho)}u$  from below by a constant depending only on $d$ and $\gamma$. The conclusion follows. 
\end{proof} 

With this, we give the proof of the main results.
\begin{proof}[Proof of Theorem~\ref{ThmMainIntro}]
With $\{g_t\}_{t\in(-1,1)}$ satisfying Assumption~\ref{Ass}  and notations from \eqref{EqnM}, we  estimate the following space-time singular set:
$$
S:=\{(x,t)\in B_1\times(-1,1):\hem x\in\Sing(u_t) \text{ for some }u_t\in\M(B_1;g_t)\}.
$$

Proposition~\ref{PropOrderingFromData} gives, for $t<s$,
$
(u_t,u_s)\in\OM(B_1).
$
Lemma~\ref{LemCleaningNew} gives 
$$C \|u_t\|_{L^\infty(B_1)} \mathrm{dist}(\Sing(u_t)\cap B_{1/2},\Sing(u_s))^{1+\gamma}\ge |t-s|$$ 
for any $\gamma<\gamma_d$ from \eqref{EqnGammaD}, where $C$ further depends on $G$ and $\theta$ in Assumption~\ref{Ass}. This establishes  condition (2) in Lemma~\ref{LemGenericReduction} for $p=1+\gamma_d$.

The desired conclusion follows from Lemma~\ref{LemGenericReduction} and Proposition~\ref{PropOptimalResultFeY}.
\end{proof} 

\begin{proof}[Proof of Corollary~\ref{CorMainIntro}]
For given boundary data $g$, we define $g_t=g+t$ on $\partial B_1$. Up to a translation in $t$, this family satisfies Assumption~\ref{Ass}. The conclusion follows from Theorem~\ref{ThmMainIntro}.
\end{proof}

%%%%%%%%%%%%%%%%%%%%%%%%%%%%%%%%%%%%%%%%%%%%%%%%%%%%%%%%%%%%%%%%%%%%%%%%%%%%%%%%%%%%%%%%%%%%%%%%%%%%%%%%%%%%%%%%%%%%%%%%%%%%%%%%%%%%%%%%%%%%%%%%%%%%%%%%%%%%%%%%%%%%%%%%%%%%%%%%%%%%%%%%%%%%%%%%%%%%%%%%%%%%%%%%%%%%%%%%%%%%%%%%%%%%%%%%%%%%%%%%%%%%%%%%%%%%%%%%%%%%
%%%%%%%%%%%%%%%%%%%%%%%%%%%%%%%%%%%%%%%%%%%%%%%%%%%

\end{document}